\pdfoutput=1
\documentclass{amsart}
\usepackage{amsthm,amsmath}
\usepackage{graphics}
\usepackage{graphicx}

\def\mathopsc#1{\mathop{\normalfont\textsc{#1}}\nolimits}

\def\outdeg{\mathopsc{Outdeg}\nolimits}
\def\indeg{\mathopsc{Indeg}\nolimits}
\def\FLOOR#1{\left\lfloor{#1}\right\rfloor}
\def\CEIL#1{\left\lceil{#1}\right\rceil}

\def\lcm{\mathop{\hbox{\rm lcm}}}

\def\head{\mathopsc{Head}}
\def\tail{\mathopsc{Tail}}

\def\In{\mathopsc{In}}
\def\Out{\mathopsc{Out}}
\def\Split{\mathopsc{Split}}

\def\CP{\mathopsc{CP}}

\def\LC{\mathcal{LC}}
\def\Cyc{\mathcal{C}}
\def\Lin{\mathcal{L}}
\def\Multicyc{\mathcal{M}} 
\def\MCDB{\mathcal{MC}} 
\def\codewords{\mathcal{W}} 
\def\nod{\mathcal{D}} 

\def\BWT{\mathopsc{BWT}}
\def\EBWT{\mathopsc{EBWT}}

\def\TB{T_B}
\def\TE{T_E}
\def\TIB{T_{IB}}
\def\TIE{T_{IE}}

\def\numoccur{\mathopsc{\#Occurrences}}

\def\linstr#1{\langle #1 \rangle}

\theoremstyle{plain}
\newtheorem{theorem}{Theorem}[section]

\newtheorem{corollary}[theorem]{Corollary}
\newtheorem{lemma}[theorem]{Lemma}

\theoremstyle{definition}
\newtheorem{example}[theorem]{Example}

\begin{document}

\title{Multi de Bruijn Sequences}
\author[Glenn Tesler]{Glenn Tesler\textsuperscript{*}}
\address{
Department of Mathematics \\
University of California, San Diego \\
La Jolla, CA 92093-0112}

\thanks{\textsuperscript{*}This work was supported in part by National Science Foundation grant CCF-1115206.}

\begin{abstract}
We generalize the notion of a de Bruijn sequence to a ``multi de Bruijn sequence'': a cyclic or linear sequence that contains every $k$-mer over an alphabet of size $q$ exactly $m$ times.
For example, over the binary alphabet $\{0,1\}$,
the cyclic sequence $(00010111)$ and the linear
sequence $000101110$ each contain two instances of each 2-mer $00,01,10,11$.
We derive formulas for the number of such sequences.
The formulas and derivation generalize classical de Bruijn sequences (the case $m=1$).
We also determine the number of multisets of aperiodic cyclic sequences containing every $k$-mer exactly $m$ times; for example,
the pair of cyclic sequences $(00011)(011)$ contains two instances of each $2$-mer listed above. This uses an extension of the Burrows-Wheeler Transform due to Mantaci et al, and generalizes a result by Higgins for the case $m=1$.
\end{abstract}

\maketitle


\section{Introduction}

We consider sequences over a totally ordered alphabet $\Omega$ of size $q\ge 1$.
A \emph{linear sequence} is an ordinary sequence $a_1,\ldots,a_n$
of elements of $\Omega$,
denoted in string notation as $a_1\ldots a_n$.
Define the \emph{cyclic shift} of a linear sequence by
$\rho(a_1 a_2 \ldots a_{n-1} a_n) = a_n a_1 a_2 \ldots a_{n-1}$.
In a \emph{cyclic sequence}, we treat all rotations of a
given linear sequence as equivalent:
$$(a_1\ldots a_n) = \{\rho^i(a_1\ldots a_n)=a_{i+1}\ldots a_n a_1\ldots a_{i-1} : i=0,\ldots,n-1 \}.$$
Each rotation $\rho^i(a_1\ldots a_n)$
 is called a \emph{linearization} of the cycle $(a_1\ldots a_n)$.

A \emph{$k$-mer} is a sequence of length $k$ over $\Omega$.
The set of all $k$-mers over $\Omega$ is $\Omega^k$.
A \emph{cyclic de Bruijn sequence}
is a cyclic sequence over alphabet $\Omega$ (of size $q$)
in which all $k$-mers
occur exactly once. The length of such a sequence is $\ell=q^k$, because each of the $q^k$ $k$-mers accounts for one starting position.

In 1894, the problem of counting cyclic de Bruijn sequences over a binary alphabet (the case $q=2$) was proposed by de Rivi\`ere~\cite{deRiviere1894} and solved by Sainte-Marie~\cite{SainteMarie1894}. In 1946, the same problem was solved by de Bruijn~\cite{deBruijn1946}, unaware of Sainte-Marie's work.
In 1951,
the solution was extended to any size alphabet ($q\ge 1$) by
van Aardenne-Ehrenfest and de Bruijn~\cite[p.~203]{vanAE1951}:
\begin{equation}
\text{\# cyclic de Bruijn sequences} = q!^{q^{k-1}}/q^k \;.
\label{eq:num_db}
\end{equation}
The sequences were subsequently named de Bruijn sequences, and work continued on them for decades before the 1894 publication by Sainte-Marie was rediscovered in 1975~\cite{deBruijn1975}.
Table~\ref{tab:notation_comparison} summarizes
the cases considered and notation used in these and other papers.

\begin{table}
\begin{tabular}{clccc}
\multicolumn{2}{l}{Reference} & Multiplicity & Alphabet size & Word size \\
\hline
\cite{deRiviere1894} & de Rivi\`ere (1894) & 1 & 2 & $n$ \\
\cite{SainteMarie1894} & Sainte-Marie (1894) & 1 & 2 & $n$ \\
\cite{deBruijn1946} & de Bruijn (1946) & 1 & 2 & $n$ \\
\cite{vanAE1951} & van A.-E. \& de Bruijn (1951) & 1 & $\sigma$ & $n$ \\
\cite{deBruijn1975} & de Bruijn (1975) & 1 & $\sigma$ & $n$ \\
\cite{DawsonGood1957} & Dawson \& Good (1957) & $k$ & $t$ & $m$ \\
\cite{Fredricksen1982} & Fredricksen (1982) & 1 & $k$  & $n$ \\
\cite{Kandel1996} & Kandel et al (1996) & variable &
	any (uses 4) & $k$ \\
\cite{Stanley_EC2_1999} & Stanley (1999) & 1 & $d$ & $n$ \\
\cite{Higgins2012} & Higgins (2012) & 1 & $k$ & $n$ \\
\cite{Osipov2016} & Osipov (2016) & $f=2$ & $\ell=2$ & $1 \le p \le 4$ \\
 & Tesler (2016) [this paper] & $m$ & $q$ & $k$
\end{tabular}

\smallskip

\caption{Notation or numerical values considered
 for de Bruijn sequences and related problems in the references.}
\label{tab:notation_comparison}
\end{table}

We introduce a \emph{cyclic multi de Bruijn sequence}: a cyclic sequence
over alphabet $\Omega$ (of size $q$)
in which all $k$-mers
 occur exactly $m$ times,
with $m,q,k\ge 1$.
Let $\Cyc(m,q,k)$ denote the set of all such sequences,
 The length of such a sequence is $\ell=mq^k$, because each of the $q^k$ $k$-mers accounts for $m$ starting positions.
For $m=2$, $q=2$, $k=3$,
$\Omega=\{0,1\}$,
one such sequence is
$
(1111011000101000)
$.
The length is
 $\ell=mq^k=2\cdot 2^3=16$.
Each $3$-mer
000, 001, 010, 011, 100, 101, 110, 111
occurs twice,
including overlapping occurrences and occurrences that wrap around the end.

Let $\LC(m,q,k)$ denote the set of linearizations of
cyclic
sequences in $\Cyc(m,q,k)$.
These will be called \emph{linearized} or \emph{linearized cyclic multi de Bruijn sequences}.
Let
$\LC_y(m,q,k)$ denote the set of linearizations that start with
 $k$-mer 
 $y\in\Omega^k$.
In the example above,
the two linearizations starting with $000$
are
$0001111011000101$ and
$0001010001111011$.

For a sequence $s$ and $i\ge 0$, let $s^i$ denote concatenating $i$ copies of $s$.

A cyclic sequence $(s)$ (or a linearization $s$)
of length $n$
has a $d$\textsuperscript{th} order rotation
for some $d|n$ (positive integer divisor $d$ of $n$) iff $\rho^{n/d}(s)=s$
iff $s=t^d$ for some sequence $t$ of length $n/d$.
The \emph{order}
of cycle $(s)$ (or linearization $s$) is
the largest $d|n$
such that
$\rho^{n/d}(s)=s$.
For example,
 $(11001100)$ has order $2$,
while if all $n$ rotations of $s$ are distinct, then $(s)$ has order $1$.
Since a cyclic multi de Bruijn sequence
has exactly $m$ copies of each $k$-mer,
the order
 must divide into $m$.
Sets $\Cyc^{(d)}(m,q,k)$, $\LC^{(d)}(m,q,k)$, and $\LC_y^{(d)}(m,q,k)$
denote multi de Bruijn sequences of order $d$
 that are cyclic, linearized cyclic, and linearized cyclic starting with $y$.

A \emph{linear multi de Bruijn sequence} is a linear sequence
of length $\ell'=mq^k+k-1$
in which all $k$-mers occur exactly $m$ times.
Let $\Lin(m,q,k)$ denote the set of such sequences
and $\Lin_y(m,q,k)$ denote those starting with
 $k$-mer $y$.
A linear sequence does not wrap around and is not the same as a linearization of a cyclic sequence.
For
$m=2$, $q=2$, $k=3$, an example is
111101100010100011,
with length $\ell'=18$.

For $m,q,k\ge 1$,
we will compute the number of
linearized cyclic (Sec.~\ref{sec:num_linearizations}),
cyclic (Sec.~\ref{sec:cyclic_mdb}), and linear (Sec.~\ref{sec:num_linear})
multi de Bruijn sequences.
We also compute the number of
cyclic and linearized cyclic sequences with a
$d$\textsuperscript{th} order rotational symmetry.
We give examples in Sec.~\ref{sec:examples}.
In Sec.~\ref{sec:bruteforce}, we
use brute force
 to generate all
multi de Bruijn sequences for small values of $m,q,k$,
confirming
our formulas in certain cases.
In Sec.~\ref{sec:random_generation}, we show how to select a uniform random linear, linearized cyclic, or cyclic multi de Bruijn sequence.
A summary of our notation and formulas is given in
Table~\ref{tab:our_notation}.

\begin{table}
\begin{tabular}{ll}
Definition			& Notation \\ \hline
Alphabet			& $\Omega$ \\
Alphabet size		& $q=|\Omega|$ \\
Word size			& $k$ \\
Multiplicity of each $k$-mer	& $m$ \\
Rotational order of a sequence	& $d$; must divide into $m$ \\
Specific $k$-mer that sequences start with	& $y$ \\
Complete de Bruijn multigraph	& $G(m,q,k)$ \\
Cyclic shift		& $\rho(a_1\ldots a_n)=a_n a_1\ldots a_{n-1}$ \\
Power of a linear sequence & $s^r=s\cdots s$ ($r$ times) \\
M\"obius function	& $\mu(n)$ \\
Euler's totient function & $\phi(n)$ \\
Set of permutations of $0^m 1^m \cdots (q-1)^m$ & $\codewords_{m,q}$ \\
\end{tabular}

\bigskip

\begin{tabular}{lll}
Multi de Bruijn Sequence		& Set			& Size \\ \hline
 && \\[-11pt]
Linear			& $\Lin(m,q,k)$		& $W(m,q,k)=\left((mq)!/m!^q\right)^{q^{k-1}}$ \\
Cyclic			& $\Cyc(m,q,k)$	& $\frac{1}{mq^k} \sum_{r|m} \phi(m/r) W(r,q,k)$ \\
Linearized cyclic	& $\LC(m,q,k)$		& $W(m,q,k)$ \\
Multicyclic			& $\MCDB(m,q,k)$	& $W(m,q,k)$ \\[10pt]
Linear, starts with $k$-mer $y$	& $\Lin_y(m,q,k)$	& $W(m,q,k)/q^k$ \\
Linearized,
 starts with $y$	& $\LC_y(m,q,k)$	& $W(m,q,k)/q^k$ \\[10pt]
Cyclic, order $d$	& $\Cyc^{(d)}(m,q,k)$	& $\frac{1}{(m/d)q^k} \sum_{r|(m/d)} \mu(r) W(\frac{m}{rd},q,k)$ \\
Linearized,
 starts with $y$, && \\
\strut\qquad order $d$
	& $\LC^{(d)}_y(m,q,k)$		&  $\frac{1}{q^k} \sum_{r|(m/d)} \mu(r) W(\frac{m}{rd},q,k)$ \\[10pt]
\end{tabular}

\smallskip

\caption{Notation and summary of formulas for the number of multi de Bruijn sequences of different types.}
\label{tab:our_notation}
\end{table}

We also consider another generalization:
a \emph{multicyclic de Bruijn sequence} is
a multiset of aperiodic cyclic sequences
such that every $k$-mer occurs exactly $m$ times among all the cycles. 
For example,
 $(00011)(011)$ has two occurrences of each $2$-mer $00,01,10,11$.
This generalizes results of Higgins~\cite{Higgins2012}
for the case $m=1$.
In Sec.~\ref{sec:MCDB_EBWT}, we develop this generalization using the ``Extended Burrows-Wheeler Transform'' of Mantaci et al~\cite{Mantaci2005,Mantaci2007}.
In Sec.~\ref{sec:MCDB_graph_cycles}, we give another method
to count multicyclic de Bruijn sequences by counting the number of ways to partition a balanced graph into
aperiodic cycles with prescribed edge multiplicities.

We implemented these formulas and algorithms in software available at \\
\verb|http://math.ucsd.edu/|$\sim$\verb|gptesler/multidebruijn|\;.

\smallskip

\emph{Related work.}
The methods van Aardenne-Ehrenfest and de Bruijn~\cite{vanAE1951} developed in 1951 to generalize de Bruijn sequences
to alphabets of any size
  potentially could have been adapted to $\Cyc(m,q,k)$; see Sec.~\ref{sec:cyclic_mdb}.
In 1957, Dawson and Good~\cite{DawsonGood1957}
counted
``circular arrays''
in which each $k$-mer occurs $m$ times (in our notation).
Their formula corresponds to an intermediate step in our solution,
but overcounts cyclic multi de Bruijn sequences;
see Sec.~\ref{sec:num_eulerian_cycles}.
Very recently, in 2016,
Osipov~\cite{Osipov2016} introduced the problem of counting ``$p$-ary $f$-fold de Bruijn sequences.''
However, Osipov only gives a partial solution, which appears to be incorrect;
see Sec.~\ref{sec:bruteforce}.

\section{Linearizations of cyclic multi de Bruijn sequences}
\label{sec:num_linearizations}

\subsection{Multi de Bruijn graph}
\label{sec:graph}
We will compute the number of cyclic multi de Bruijn sequences by counting Eulerian cycles in a graph and adjusting for multiple Eulerian cycles corresponding to each cyclic multi de Bruijn sequence.
For now, we assume that $k\ge 2$; we will separately consider $k=1$ in Sec.~\ref{sec:k=1}.
We will also clarify details for the case $q=1$ in Sec.~\ref{sec:q=1}.

Define a multigraph $G=G(m,q,k)$ whose vertices are the $(k-1)$-mers over
alphabet $\Omega$ of size $q$.
There are $n=q^{k-1}$ vertices.
For each $k$-mer $y=c_1c_2\ldots c_k\in\Omega^k$, add $m$ directed edges
$c_1c_2\ldots c_{k-1} \mathop{\to}\limits^y c_2\ldots c_{k-1}c_k$,
each labelled by $y$.
Every vertex has outdegree $mq$ and indegree $mq$.
Further, we will show
in Sec.~\ref{sec:matrices}
 that the graph is strongly connected,
so $G$ is Eulerian.

Consider a walk through the graph:
$$
c_1 c_2 \cdots c_{k-1}
\;\to\;
c_2 c_3 \cdots c_k
\;\to\;
c_3 c_4 \cdots c_{k+1}
\;\to\;
\cdots
\;\to\;
c_r c_{r+1} \cdots c_{r+k-2}
$$
This walk corresponds to a linear sequence
$c_1 c_2 \cdots c_{r+k-2}$,
which can be determined either from vertex labels (when $k\ge 2$) or edge labels (when $k\ge 1$).
If the walk is a cycle (first vertex equals last vertex), then it also
corresponds to
a linearization $c_1 c_2 \cdots c_{r-1}$
of cyclic sequence $(c_1 c_2 \cdots c_{r-1})$.
Starting at another location in the cycle will result in a different linearization of the same cyclic sequence.

Any Eulerian cycle in this graph
gives a cyclic sequence
in $\Cyc(m,q,k)$.
Conversely, each sequence in $\Cyc(m,q,k)$
corresponds to
at least one Eulerian cycle.

Cyclic multi de Bruijn sequences
may also be obtained through
a generalization of Hamiltonian cycles.
Consider cycles in $G(1,q,k+1)$,
with repeated edges and vertices allowed,
in which every vertex ($k$-mer) is visited exactly $m$ times
(not double-counting the initial vertex when it is used to close the cycle).
Each such graph cycle
starting at vertex $y\in\Omega^k$
corresponds to the linearization it spells in
$\LC_y(m,q,k)$, and vice-versa.
However, we will
 use
 the Eulerian cycle approach because it leads to an enumeration formula.

\subsection{Matrices}
\label{sec:matrices}
Let $n=q^{k-1}$ and
form the $n \times n$ adjacency matrix $A$ of directed graph $G=G(m,q,k)$.
When $k\ge 2$,
\begin{equation}
A_{c_1\ldots c_{k-1},d_1\ldots d_{k-1}} =
\begin{cases}
m & \text{if $c_2\ldots c_{k-1}=d_1\ldots d_{k-2}$} \\
0 & \text{otherwise.}
\end{cases}
\label{eq:Avw}
\end{equation}

For every pair of $(k-1)$-mers
$v=c_1\ldots c_{k-1}$ and $w=d_1\ldots d_{k-1}$,
the walks of length $k-1$ from $v$ to $w$ have this form:
$$c_1\ldots c_{k-1}
\;\to\; c_2\ldots c_{k-1} d_1
\;\to\; \cdots
\;\to\; c_{k-1} d_1\ldots d_{k-2}
\;\to\; d_1 \ldots d_{k-1}
$$
For each of the $k-1$ arrows, there are $m$ parallel edges to choose from.
Thus, there are $m^{k-1}$ walks of length $k-1$ from $v$ to $w$,
so the graph is strongly connected and
$(A^{k-1})_{v,w} = m^{k-1}$ for all pairs of vertices $v,w$.
This gives $A^{k-1} = m^{k-1} J$, where $J$ is the $n \times n$ matrix of all 1's.

$J$ has an eigenvalue $n$ of multiplicity $1$
and an eigenvalue $0$ of multiplicity $n-1$.
So $A^{k-1}$ has
an eigenvalue $m^{k-1}n=(mq)^{k-1}$ of multiplicity $1$
and an eigenvalue $0$ of multiplicity $n-1=q^{k-1}-1$.

Thus, $A$ has one eigenvalue of the form $mq\omega$, where $\omega$ is a $(k-1)$\textsuperscript{th} root of unity, and $q^{k-1}-1$ eigenvalues equal to $0$.
In fact, the first eigenvalue of $A$ is $mq$, because the all 1's column vector $\vec 1$ is a right eigenvector with eigenvalue $mq$
(that is, $A\vec 1 = mq \, \vec 1$):
for every vertex $w\in\Omega^{k-1}$, we have
$$
(A\vec 1)_w
= \sum_{v\in\Omega^{k-1}} (A_{vw}) (\vec1)_w
= \sum_{v\in\Omega^{k-1}} A_{vw} \cdot 1
= \indeg(w) = mq \;.
$$
Similarly, the all 1's row vector is a left eigenvector with eigenvalue $mq$.

The degree matrix of an Eulerian graph is an $n\times n$ diagonal matrix with
$D_{vv}=\indeg(v)=\outdeg(v)$ on the diagonal and
$D_{vw}=0$ for $v\ne w$.
All vertices in $G$ have
indegree and outdegree $mq$, so $D=mqI$.
The Laplacian matrix of $G$ is $L = D-A = mqI - A$.
It has one eigenvalue equal to $0$ and $q^{k-1}-1$ eigenvalues equal to $mq$.

\subsection{Number of Eulerian cycles in the graph}
\label{sec:num_eulerian_cycles}
Choose any edge $e=(v,w)$ in $G$. Let $y$ be the $k$-mer represented by $e$.
The number of spanning trees of $G$ with root $v$ (with a directed path from each vertex to root $v$) is given by
Tutte's Matrix-Tree Theorem for directed graphs~\cite[Theorem~3.6]{Tutte1948}
(also see~\cite[Theorem~7]{vanAE1951} and~\cite[Theorem~5.6.4]{Stanley_EC2_1999}).
For a directed Eulerian graph, the formula can be expressed as
follows~\cite[Cor.~5.6.6]{Stanley_EC2_1999}:
the number of spanning trees of $G$ rooted at $v$ is
\begin{equation}
\frac{1}{n} \cdot \left(\text{product of the $n-1$ nonzero eigenvalues of $L$}\right)
= \frac{(mq)^{q^{k-1}-1}}{q^{k-1}} \;.
\label{eq:num_spanning_trees}
\end{equation}

By the BEST Theorem~\cite[Theorem~5b]{vanAE1951} (also see~\cite[pp. 56, 68]{Stanley_EC2_1999} and~\cite{Fredricksen1982}), the number of Eulerian cycles with initial edge $e$ is
\begin{align}
\text{(\# spanning} & \text{~trees rooted at $v$)}
\cdot \prod_{x\in V} (\outdeg(x) - 1)!
\label{eq:BEST_product}
\\
&=
 \frac{1}{q^{k-1}} \cdot (mq)^{q^{k-1}-1}  \cdot  (mq-1)!^{q^{k-1}}
=
 \frac{1}{m \cdot q^k}  \cdot  (mq)!^{q^{k-1}} \;.
\label{eq:num_eulerian_cycles}
\end{align}

Each Eulerian cycle spells out a linearized multi de Bruijn sequence that
starts with $k$-mer $y$.
However, when $m\ge 2$,
there are multiple cycles generating each such sequence.
Let $C$ be an Eulerian cycle spelling out linearization $s$.
For $k$-mer $y$,
the first edge of the cycle, $e$, was given;
the other $m-1$ edges labelled by $y$ are parallel to $e$ and may be 
permuted in $C$ in any of $(m-1)!$ ways.
For the other $q^k-1$ $k$-mers, the $m$ edges representing each $k$-mer can be permuted in $C$ in $m!$ ways.
Thus, each linearization $s$ starting in $y$ is generated by
$(m-1)!\cdot m!^{q^k-1}=(m!^{q^k})/m$ Eulerian cycles that start with edge $e$.
Divide Eq.~(\ref{eq:num_eulerian_cycles}) by this factor
to obtain the number of linearized multi de Bruijn sequences starting with $k$-mer $y$:
\begin{align}
|\LC_y(m,q,k)|
&=
 \frac{1}{m \cdot q^k}  \cdot  \frac{(mq)!^{q^{k-1}}}{m!^{q^k} / m}
=  \frac{1}{q^k}  \cdot \frac{(mq)!^{q^{k-1}}}{m!^{q^k}}
 = \frac{W(m,q,k)}{q^k}
\label{eq:S_y(m,q,k)}
\end{align}
where we set
\begin{align}
W(m,q,k) &=
\frac{(mq)!^{q^{k-1}}}{m!^{q^k}}
= \left(\frac{(mq)!}{m!^q}\right)^{q^{k-1}}
= \binom{mq}{\;\underbrace{m,\ldots,m}_{\text{$q$ of these}}\;}^{q^{k-1}}
\;.
\label{eq:W(m,q,k)}
\end{align}

\emph{Related work.}
Dawson and Good~\cite[p.~955]{DawsonGood1957}
computed the number of ``circular arrays'' containing
each $k$-mer exactly $m$ times (in our notation; theirs is different),
and obtained an answer equivalent 
to~(\ref{eq:num_eulerian_cycles}).
They counted graph cycles in $G(m,q,k)$ starting on a
specific edge, but when $m\ge 2$, this does not equal
$|\Cyc(m,q,k)|$;
 their count includes each cyclic multi de Bruijn sequence multiple times.
This is because they use the convention that the
multiple occurrences of each symbol
 are distinguishable~\cite[p.~947]{DawsonGood1957}.
We do additional steps to count each cyclic multi de Bruijn sequence just once: we adjust for multiple graph cycles corresponding to each linearization~(\ref{eq:S_y(m,q,k)}) and for multiple linearizations of each cyclic sequence (Sec.~\ref{sec:cyclic_mdb}).

\subsection{Special case $k=1$}
\label{sec:k=1}
When $k=1$,
the number of sequences of length $mq^k=mq$ with each symbol in $\Omega$
occurring exactly $m$ times is
$(mq)!/m!^q$.
Divide by $q$ to obtain the number of
linearized de Bruijn sequences 
starting with
each $1$-mer $y\in\Omega$:
$$|\LC_y(m,q,1)|=\frac{1}{q} \cdot \frac{(mq)!}{m!^q} \;.$$
This agrees with plugging $k=1$ into~(\ref{eq:S_y(m,q,k)}),
so~(\ref{eq:S_y(m,q,k)}) holds for $k=1$ as well.

The derivation of~(\ref{eq:S_y(m,q,k)})
in Secs.~\ref{sec:graph}--\ref{sec:num_eulerian_cycles}
requires $k\ge 2$, due to technicalities. The above
derivation uses a different method, but we can also adapt the first derivation as follows.

$G(m,q,1)$ has a single vertex `$\emptyset$' (a $0$-mer) and $mq$ loops 
on $\emptyset$, 
because for each  $i\in\Omega$, there are $m$ loops
$\emptyset\mathop{\to}\limits^i\emptyset$.
The sequence spelling out a walk can be determined from edge labels but not vertex labels, since they're null.

The adjacency matrix is $A=[mq]$ rather than Eq.~(\ref{eq:Avw}).
The Laplacian matrix $L=mqI-A=[0]$ has one eigenvalue: $0$.
In~(\ref{eq:num_spanning_trees}),
the product of nonzero eigenvalues is vacuous;
on plugging in $k=1$, the formula correctly gives that there is $1$ spanning tree (it is the vertex $\emptyset$ without any edges).
The remaining steps work for $k=1$ as-is.

\subsection{Special case $q=1$}
\label{sec:q=1}

For $q=1$,
the only cyclic sequence of length $\ell=mq^k=m\cdot 1^k=m$ is $(0^m)$.
We regard cycle $(0^m)$ and linearization $0^m$
as having an occurrence of $0^k$ starting at each position, even if $k>m$, in which case some positions of $(0^m)$ are used
multiple times
 to form $0^k$.
With this convention, there are
exactly $m$ occurrences of $0^k$
in
$(0^m)$.
Thus, for $q=1$ and $m,k\ge 1$,
\begin{align*}
\LC(m,1,k)&=\{0^m\},
& \LC_{0^k}(m,1,k)&=\{0^m\}, \\
\Cyc(m,1,k)&=\{(0^m)\}, 
& \Lin(m,1,k)&=\{0^{m+k-1}\} .
\end{align*}
Using positions of a cycle multiple times to form a $k$-mer will also arise in multicyclic de Bruijn sequences in Sec.~\ref{sec:MCDB_EBWT}.

\section{Cyclic multi de Bruijn sequences}
\label{sec:cyclic_mdb}

In a classical de Bruijn sequence ($m=1$),
each $k$-mer occurs exactly once, so
the number of
cyclic de Bruijn sequences equals
the number of linearizations that begin with any fixed $k$-mer $y$:
$$|\Cyc(1,q,k)|=|\LC_y(1,q,k)| = (q!)^{q^{k-1}} / q^k \;.$$
This agrees with~(\ref{eq:num_db}).
But a multi de Bruijn sequence
has $m$ cyclic shifts starting with $y$,
and the number of these that are distinct varies by sequence.
Let $d|m$ be the order of the cyclic symmetry of the sequence;
then there are $m/d$
distinct linearizations starting with $y$, so
\begin{equation}
|\LC^{(d)}_y(m,q,k)|=\frac{m}{d} \cdot  |\Cyc^{(d)}(m,q,k)|
\;.
\end{equation}
Further, we have the following:
\begin{lemma} For $m,q,k\ge 1$ and $d|m$:

(a)~$|\LC^{(d)}(m,q,k)|=|\LC^{(1)}(m/d,q,k)|$

(b)~For each $k$-mer $y\in\Omega^k$,
$|\LC^{(d)}_y(m,q,k)|=|\LC^{(1)}_y(m/d,q,k)|$.

(c)~$|\Cyc^{(d)}(m,q,k)|=|\Cyc^{(1)}(m/d,q,k)|$
\label{lem:m/d}
\end{lemma}
\begin{proof}
At $q=1$, all sets in (a)--(c) have size $1$.
Below, we consider $q\ge 2$.

(a)~Let $s\in\LC^{(d)}(m,q,k)$.
Since $s$ has order $d$, it splits into $d$ equal parts, $s=t^d$.
 The $m$ occurrences of each $k$-mer in $(s)$ reduce to $m/d$ occurrences of each $k$-mer in $(t)$,
and the order of $t$
is
 $d/d=1$,
so $t\in\LC^{(1)}(m/d,q,k)$.

Conversely, if $t\in\LC^{(1)}(m/d,q,k)$ then $(t^d)$ has $d \cdot (m/d)=m$ occurrences of each $k$-mer and has order $d$, so $(t^d)\in\LC^{(d)}(m,q,k)$.

(b)~Continuing with (a),
the length of $t$ is at least $k$
(since $|t|=q^k\ge k$ for $q\ge 2$ and $k\ge 1$), so the initial $k$-mer of
$s$ and $t$ must be the same.

(c)~The map in (a) induces a bijection
$f:\Cyc^{(d)}(m,q,k)\to\Cyc^{(1)}(m/d,q,k)$ as follows.
Let $\sigma\in\Cyc^{(d)}(m,q,k)$.
For any linearization $s$ of $\sigma$, let $s=t^d$
and set $f(\sigma)=(t)$.
This is well-defined:
the distinct linearizations of $\sigma$ are
$\rho^i(s)$ where $i=0,\ldots,(mq^k/d)-1$.
Rotating $s$ by $i$ also rotates $t$ by $i$ and gives
the same
cycle in $\Cyc^{(1)}(m/d,q,k)$, so all linearizations of $\sigma$ give
the same result.

Conversely, given $(t)\in\Cyc^{(1)}(m/d,q,k)$, the unique inverse
is $f^{-1}((t))=(t^d)$. As above, this is well-defined.
\end{proof}

Partitioning the cyclic multi de Bruijn sequences by order gives
the following, where the sums run over positive integers $d$ that divide into $m$:
\begin{equation}
|\Cyc(m,q,k)| = \sum_{d|m} |\Cyc^{(d)}(m,q,k)| = \sum_{d|m} |\Cyc^{(1)}(m/d,q,k)| \;,
\label{eq:Bcirc1}
\end{equation}

Each cyclic sequence of order $d$ has $m/d$ distinct linearizations starting with $y$, so
\begin{align*}
|\LC_y(m,q,k)|
&= \sum_{d|m} |\LC^{(d)}_y(m,q,k)|
= \sum_{d|m} \frac{m}{d} \cdot |\Cyc^{(d)}(m,q,k)|
\\
&= \sum_{d|m} \frac{m}{d} \cdot |\Cyc^{(1)}(m/d,q,k)|
= \sum_{d|m} d \cdot |\Cyc^{(1)}(d,q,k)| \;.
\end{align*}
By M\"obius Inversion,
$$
m\cdot |\Cyc^{(1)}(m,q,k)| = \sum_{d|m} \mu(d) \cdot |\LC_y(m/d, q, k)|
\;.
$$
Solve this for $|\Cyc^{(1)}(m,q,k)|$:
\begin{align}
|\Cyc^{(1)}(m,q,k)| &= \frac{1}{m}\sum_{d|m} \mu(d) \cdot |\LC_y(m/d, q, k)|
 \;.
\label{eq:R_1b}
\end{align}
Combining this with~(\ref{eq:S_y(m,q,k)}) and Lemma~\ref{lem:m/d}(c) gives
\begin{align}
|\Cyc^{(d)}(m,q,k)|=
\frac{1}{(m/d)q^k} \sum_{r|(m/d)} \mu(r) W(m/(rd),q,k) \;.
\end{align}
Plug~(\ref{eq:R_1b}) into~(\ref{eq:Bcirc1}) to obtain
\begin{align*}
|\Cyc(m,q,k)|
&= \sum_{d|m} |\Cyc^{(1)}(m/d,q,k)|
\\
&= \sum_{d|m} \frac{1}{m/d} \sum_{d'|(m/d)}  \mu(d') \cdot |\LC_y(m/(dd'), q, k)| \;.
\\
\noalign{\noindent Change variables: set $r=dd'$}
&= \sum_{r|m} \sum_{d|r} \frac{1}{m/d} \cdot \mu(r/d) \cdot |\LC_y(m/r, q, k)|
\\
&= \frac{1}{m} \sum_{r|m}
 \left(\sum_{d|r} d \cdot \mu(r/d)\right)
 \cdot
 |\LC_y(m/r, q, k)|
\\
& = \frac{1}{m} \sum_{r|m} \phi(r) \cdot |\LC_y(m/r, q, k)|
= \frac{1}{m} \sum_{r|m} \phi(m/r) \cdot |\LC_y(r, q, k)|
 \;,
\end{align*}
where $\phi(r)$ is the Euler totient function.
Thus,
\begin{equation}
|\Cyc(m,q,k)| =
\frac{1}{mq^k} \sum_{r|m}
\phi(m/r) \cdot 
W(r,q,k) \;.
\label{eq:Bcirc(m,q,k)}
\end{equation}

\emph{Related work.}
Van Aardenne-Ehrenfest and de Bruijn~\cite[Sec. 5]{vanAE1951} take
a directed Eulerian multigraph $G$ and replace each edge by a ``bundle'' of $\lambda$ parallel edges to obtain a multigraph $G^\lambda$. They compute the number of Eulerian cycles in $G^\lambda$ in terms of the number of Eulerian cycles in $G$,
and obtain a formula related to~(\ref{eq:Bcirc(m,q,k)}).
Their bundle method could have been applied to count cyclic multi de Bruijn
sequences in terms of ordinary cyclic de Bruijn sequences
by taking $G=G(1,q,k)$, $\lambda=m$, and considering $G^m$,
but they did not do this.
Instead, they set $\lambda=q$ and
found a correspondence between
Eulerian cycles in
$G(1,q,k+1)$ vs.\  $G^q$ with $G=G(1,q,k)$,
yielding a recursion in $k$ for $|\Cyc(1,q,k)|$.
They derived~(\ref{eq:num_db}) by this recursion rather than
the BEST Theorem, which was also introduced in that paper.
Dawson and Good~\cite[p.~955]{DawsonGood1957} subsequently derived~(\ref{eq:num_db}) by the BEST Theorem.

\section{Linear multi de Bruijn sequences}
\label{sec:num_linear}
Now we consider 
$\Lin(m,q,k)$, the set of
\emph{linear sequences} in which every $k$-mer over $\Omega$ occurs exactly $m$ times.
A linear sequence is not the same as a linearized representation of a circular sequence.
Recall that a linearized sequence has length $\ell=mq^k$.
A linear sequence has length $\ell'=mq^k + k-1$, because
there are $mq^k$ positions at which $k$-mers start, and an additional $k-1$ positions at the end to complete the final $k$-mers. Below, assume $q\ge 2$.

\begin{lemma}
For $q\ge 2$, for each sequence in $\Lin(m,q,k)$,
the first and last $(k-1)$-mer are the same.
\end{lemma}
\begin{proof}
Count the number of times each $(k-1)$-mer $x\in\Omega^{k-1}$ occurs in
$s'=a_1 a_2 \ldots a_{\ell'}\in\Lin(m,q,k)$
as follows:

Each $k$-mer occurs $m$ times in $s'$ among starting positions $1,\ldots,mq^k$. Since
$x$
 is
 the first $k-1$ letters of $q$ different $k$-mers, there are $mq$ occurrences of $x$ starting in this range.
There is also a $(k-1)$-mer $x_1$ with an additional occurrence at starting position $mq^k+1$ (the last $k-1$ positions of $s'$), so $x_1$ occurs $mq+1$ times in $s'$ while all other $(k-1)$-mers occur $mq$ times.

By a similar argument, the $(k-1)$-mer $x_2$
at the start of $s'$
has a total of $mq+1$ occurrences in $s'$, and all other $(k-1)$-mers
occur exactly $mq$ times.

So $x_1$ and $x_2$ each occur $mq+1$ times in $s'$.
But we showed that there is only one $(k-1)$-mer occurring $mq+1$ times, so $x_1=x_2$.
\end{proof}

For $q\ge 2$, this leads to
a bijection
$f:\Lin(m,q,k)\to\LC(m,q,k)$.
Given $s'\in\Lin(m,q,k)$, drop the last $k-1$ characters to obtain
sequence $s=f(s')$.
To invert, given $s\in\LC(m,q,k)$, form $s'=f^{-1}(s)$ by repeating
the first $k-1$ characters of $s$ at the end.

For $i\in\{1,\ldots,\ell\}$,
the same $k$-mer occurs linearly at position $i$ of $s'$ and
circularly at position $i$ of $s$,
so every $k$-mer occurs the same number of times in
both linear sequence $s'$
and
cycle $(s)$.
Thus, $f$ maps $\Lin(m,q,k)$ to $\LC(m,q,k)$, and $f^{-1}$ does the reverse.

This also gives $|\Lin_y(m,q,k)|=|\LC_y(m,q,k)|$.
Multiply~(\ref{eq:S_y(m,q,k)}) by $q^k$ choices of initial $k$-mer $y$ to obtain
\begin{equation}
|\Lin(m,q,k)|
=|\LC(m,q,k)|
= q^k \cdot |\LC_y(m,q,k)|
=
W(m,q,k)
\;.
\label{eq:Blin(m,q,k)}
\end{equation}

Above, we assumed $q\ge 2$.
Eq.~(\ref{eq:Blin(m,q,k)})
also holds for $q=1$, since all four parts equal $1$
(see Sec.~\ref{sec:q=1}).

\section{Examples}
\label{sec:examples}

\begin{table}
(A)~$\LC_{00}(m=2,q=2,k=2)$: Linearizations starting with 00

\smallskip

\begin{tabular}{lllll}
Order 1: & 00010111 & 00011011 & 00011101 & 00100111 \\
              & 00101110 & 00110110 & 00111010 & 00111001 \\
Order 2: & 00110011
\end{tabular}

\bigskip

(B)~$\Cyc(m=2,q=2,k=2)$: Cyclic multi de Bruijn sequences

\smallskip

\begin{tabular}{lllll}
Order 1: & (00010111) & (00011011) & (00011101) & (00100111)
\\
Order 2: & (00110011)
\end{tabular}

\bigskip

(C)~$\Lin(m=2,q=2,k=2)$: Linear multi de Bruijn sequences

\smallskip

\begin{tabular}{llll}
 000101110 &  010001110 &  100010111 &  110001011 \\
 000110110 &  010011100 &  100011011 &  110001101 \\
 000111010 &  010111000 &  100011101 &  110010011 \\
 001001110 &  011000110 &  100100111 &  110011001 \\
 001011100 &  011001100 &  100110011 &  110100011 \\
 001100110 &  011011000 &  100111001 &  110110001 \\
 001101100 &  011100010 &  101000111 &  111000101 \\
 001110010 &  011100100 &  101100011 &  111001001 \\
 001110100 &  011101000 &  101110001 &  111010001
\end{tabular}

\bigskip

(D)~$\MCDB(m=2,q=2,k=2)$: Multicyclic de Bruijn sequences

\smallskip

\begin{tabular}{llll}
(0)(0)(01)(01)(1)(1)  &	(0)(001011)(1)    & (0001)(01)(1)(1)   & (00011)(01)(1)   \\
(0)(0)(01)(011)(1)    & (0)(0010111)      & (0001)(011)(1)     & (000111)(01)     \\
(0)(0)(01)(0111)      & (0)(0011)(011)    & (0001)(0111)       & (00011011)       \\
(0)(0)(01011)(1)      & (0)(00101)(1)(1)  & (0001011)(1)       & (001)(001)(1)(1) \\
(0)(0)(010111)        & (0)(001101)(1)    & (00010111)         & (001)(0011)(1)   \\
(0)(0)(011)(011)      & (0)(0011101)      & (00011)(011)       & (001)(00111)     \\
(0)(001)(01)(1)(1)    & (0)(0011)(01)(1)  & (000101)(1)(1)     & (0010011)(1)     \\
(0)(001)(011)(1)      & (0)(00111)(01)    & (0001101)(1)       & (00100111)       \\
(0)(001)(0111)        & (0)(0011011)      & (00011101)         & (0011)(0011)
\end{tabular}

\bigskip

\caption{Multi de Bruijn sequences of each type, with $m=2$ copies of each $k$-mer, 
alphabet $\Omega=\{0,1\}$ of size $q=2$,
and word size $k=2$.
}
\label{tab:(m,q,k)=(2,2,2)}
\end{table}

The case $m=1$, arbitrary $q,k\ge1$, is the classical de Bruijn sequence.
Eq.~(\ref{eq:Bcirc(m,q,k)}) has just one term, $r=1$, and agrees with~(\ref{eq:num_db}):
$$
|\Cyc(1,q,k)|
= \frac{1}{q^k} \cdot \phi(1) \cdot \binom{q}{1,\ldots,1}^{q^{k-1}}
= \frac{1}{q^k} \cdot 1 \cdot q!^{q^{k-1}}
= \frac{q!^{q^{k-1}}}{q^k} \;.
$$

When $m=p$ is prime, the sum in~(\ref{eq:Bcirc(m,q,k)})
for the number of cyclic multi de Bruijn sequences 
has two terms (divisors $r=1,p$):
\begin{align}
|\Cyc(p,q,k)|
&= \frac{1}{pq^k}
\Bigl(
\phi(1) \cdot
W(p,q,k)
+
\phi(p) \cdot
W(1,q,k)
\Bigr)
\nonumber
\\
&= \frac{1}{pq^k}
\left(
\left(\frac{(pq)!}{p!^q}\right)
^{q^{k-1}}
+
(p-1) \cdot (q!)^{q^{k-1}}
\right) .
\label{eq:Bcirc_m_prime}
\end{align}

For $(m,q,k)=(2,2,2)$,
Table~\ref{tab:(m,q,k)=(2,2,2)} lists all
multi de Bruijn sequences of each type.
Table~\ref{tab:(m,q,k)=(2,2,2)}(A)
shows $9$ linearizations beginning with $y=00$.
By~(\ref{eq:S_y(m,q,k)}),
$$
|\LC_{00}(2,2,2)|
= \frac{1}{2^2} \binom{2\cdot 2}{2,2}^{2^{2-1}}
= \frac{1}{4} \binom{4}{2,2}^{2}
= \frac{1}{4} \cdot 6^2 = 9 \;.
$$
Some of these are equivalent upon cyclic rotations, yielding 5 distinct cyclic multi de Bruijn sequences,
shown in Table~\ref{tab:(m,q,k)=(2,2,2)}(B).
By~(\ref{eq:Bcirc(m,q,k)}) or~(\ref{eq:Bcirc_m_prime}),
\begin{align*}
|\Cyc(2,2,2)|
&= \frac{1}{2\cdot 2^2} \sum_{r|2} \phi(2/r) \binom{2r}{r,r}^{2^1}
= \frac{1}{8}
\left(
 \phi(1)\binom{2\cdot 2}{2,2}^{2}
 +
 \phi(2) \binom{2\cdot 1}{1,1}^{2}
 \right)
 \\
&= \frac{1}{8}
\left(
 1\cdot\binom{4}{2,2}^{2}
 +
 1\cdot\binom{2}{1,1}^{2}
 \right)
= \frac{1}{8}
\left(
 6^2
 +
 2^2
 \right)
= \frac{40}{8} = 5 \;.
\end{align*}

There are 9 linearizations starting with each $2$-mer 00, 01, 10, 11. Converting these to linear multi de Bruijn sequences yields $36$ linear multi de Bruijn sequences in total, shown in Table~\ref{tab:(m,q,k)=(2,2,2)}(C). By~(\ref{eq:Blin(m,q,k)}),
$$
|\Lin(2,2,2)|=\binom{2\cdot 2}{2,2}^{2^1}=\binom{4}{2,2}^2=6^2=36 .
$$
Table~\ref{tab:(m,q,k)=(2,2,2)}(D) shows the 36 multicyclic de Bruijn sequences;
 see Sec.~\ref{sec:MCDB_EBWT}.

\section{Generating multi de Bruijn sequences by brute force}
\label{sec:bruteforce}

Osipov~\cite{Osipov2016} recently
defined a ``$p$-ary $f$-fold de Bruijn sequence''; this is the same as a cyclic multi de Bruijn sequence but in different terminology and notation.
The description below is in terms of our notation $m,q,k$,
which respectively correspond to Osipov's $f,\ell,p$
(see Table~\ref{tab:our_notation}).

Osipov developed a
method~\cite[Prop.~3]{Osipov2016}
to compute the number of cyclic multi de Bruijn sequences for 
multiplicity $m=2$, alphabet size $q=2$,
and word size $k\ge 1$.
This method requires performing a detailed construction for each $k$.
The paper states the two sequences for $k=1$, which agrees with
our $\Cyc(2,2,1)=\{(0011),(0101)\}$ of size 2,
and then carries out this construction to evaluate cases $k=2,3,4$.
It determines complete answers for those specific cases only~\cite[pp.~158--161]{Osipov2016}.
For $k=2$, Osipov's answer agrees with ours, $|\Cyc(2,2,2)|=5$.
But for $k=3$ and $k=4$ respectively, Osipov's method gives 72 and 43768,
whereas
by~(\ref{eq:Bcirc(m,q,k)}) or (\ref{eq:Bcirc_m_prime}), we obtain
$|\Cyc(2,2,3)|=82$ and $|\Cyc(2,2,4)|=52496$.

For these parameters, it is feasible to find all cyclic multi de Bruijn sequences
by a brute force search in
$\Omega^\ell$ (with $\ell=mq^k$).
This agrees with our
formula~(\ref{eq:Bcirc(m,q,k)}),
while disagreeing with Osipov's results for $k=3$ and $4$.
Table~\ref{tab:(m,q,k)=(2,2,3)} lists the 82 sequences
in $\Cyc(2,2,3)$.
The 52496 sequences for $\Cyc(2,2,4)$
and an implementation of the algorithm below
 are on the website listed in the introduction.
We also used brute force to confirm~(\ref{eq:Bcirc(m,q,k)}) for all combinations of
parameters $m,q,k\ge 2$ with at most a 32-bit search space
($q^\ell \le 2^{32}$):
\begin{align*}
|\Cyc(2,2,2)| &= 5 		& |\Cyc(6,2,2)| &= 35594 	& |\Cyc(2,2,3)| &= 82 \\
|\Cyc(3,2,2)| &= 34		& |\Cyc(7,2,2)| &= 420666	& |\Cyc(3,2,3)| &= 6668 \\
|\Cyc(4,2,2)| &= 309		& |\Cyc(8,2,2)| &= 5176309	& |\Cyc(4,2,3)| &= 750354 \\
|\Cyc(5,2,2)| &= 3176	& |\Cyc(2,3,2)| &= 40512	& |\Cyc(2,2,4)| &= 52496.
\end{align*}

Represent $\Omega^\ell$ by $\ell$-digit base $q$ integers (range
$N\in[0,q^\ell-1]$):
$$
N=\sum_{j=0}^{\ell-1} a_j q^j
\quad
(0\le a_j < q)
\quad
\text{represents~}
s=a_{\ell-1} a_{\ell-2} \ldots a_0 \;.
$$

Since we represent each element of $\Cyc(m,q,k)$ by a linearization 
that starts with $k$-mer $0^k$, 
we restrict to $N\in[0,q^{\ell-k+1}-1]$.
To generate $\Cyc(m,q,k)$:
\begin{itemize}
\item[(B1)]
for $N$ from $0$ to $q^{\ell-k+1}-1$ (pruning as described later):
\item[(B2)] \strut\qquad let $s=a_{\ell-1} a_{\ell-2} \ldots a_0$ be the $\ell$-digit base $q$ expansion of $N$
\item[(B3)] \strut\qquad if (cycle $(s)$ is a cyclic multi de Bruijn sequence
\item[(B4)] \strut\qquad \phantom{if~(}and $s$ is
smallest
 among its $\ell$ cyclic shifts)
\item[(B5)] \strut\qquad
then output cycle $(s)$
\end{itemize}

To generate $\LC_{0^k}(m,q,k)$,
skip (B4) and output
linearizations 
$s$ in (B5).

To determine if $(s)$ is a cyclic multi de Bruijn sequence, examine its $k$-mers
$a_{i+k-1} a_{i+k-2} \ldots a_i$ (subscripts taken modulo $\ell$)
in the order $i=\ell-k, \ell-k-1, \ldots, 1-k$,
counting how many times each $k$-mer is encountered.
Terminate the test early (and potentially prune as described below) if any $k$-mer count exceeds $m$.
If the test does not terminate early, then all $k$-mer counts must equal $m$, and it is a multi de Bruijn sequence.

If $i>0$ when $k$-mer counting terminates,
then 
some $k$-mer count exceeded $m$ without examining
$a_{i-1}\ldots a_0$.
Advance $N$ to $\left(\FLOOR{N/q^i} + 1\right)q^i$ instead of $N+1$,
skipping all remaining sequences
starting
 with the same
$\ell-i$ characters.
This is equivalent to incrementing the $\ell-i$ digit prefix
$a_{\ell-1}\ldots a_i$ as a base $q$ number and then concatenating $i$ $0$'s to the end.
However, if $i\le 0$,
 then all characters have been examined, so advance $N$ to $N+1$.

We may do additional pruning
by counting $k'$-mers for $k'\in\{1,\ldots,k-1\}$.
Each $k'$-mer $x$ is a prefix of $q^{k-k'}$ $k$-mers and thus should occur $mq^{k-k'}$ times in $(s)$.
If a $k'$-mer has more than $mq^{k-k'}$ occurrences in $a_{\ell-1}\cdots a_i$ with $i>0$, advance $N$ as described above.

We may also restrict the search space to length $\ell$ sequences with
exactly 
$mq^{k-1}$ occurrences of each of the $q$ characters.
We did not need to implement this for the parameters considered.

\begin{table}
\begin{tabular}{l}
(0000100101101111), (0000100101110111), (0000100101111011), \\
(0000100110101111), (0000100110111101), (0000100111010111), \\
(0000100111011101), (0000100111101011), (0000100111101101), \\
(0000101001101111), (0000101001110111), (0000101001111011), \\
(0000101011001111), (0000101011100111), (0000101011110011), \\
(0000101100101111), (0000101100111101), (0000101101001111), \\
(0000101101111001), (0000101110010111), (0000101110011101), \\
(0000101110100111), (0000101110111001), (0000101111001011), \\
(0000101111001101), (0000101111010011), (0000101111011001), \\
(0000110010101111), (0000110010111101), (0000110011110101), \\
(0000110100101111), (0000110100111101), (0000110101001111), \\
(0000110101111001), (0000110111100101), (0000110111101001), \\
(0000111001010111), (0000111001011101), (0000111001110101), \\
(0000111010010111), (0000111010011101), (0000111010100111), \\
(0000111010111001), (0000111011100101), (0000111011101001), \\
(0000111100101011), (0000111100101101), (0000111100110101), \\
(0000111101001011), (0000111101001101), (0000111101010011), \\
(0000111101011001), (0000111101100101), (0000111101101001), \\
(0001000101101111), (0001000101110111), (0001000101111011), \\
(0001000110101111), (0001000110111101), (0001000111010111), \\
(0001000111011101), (0001000111101011), (0001000111101101), \\
(0001010001101111), (0001010001110111), (0001010001111011), \\
(0001010110001111), (0001010111000111), (0001010111100011), \\
(0001011000101111), (0001011000111101), (0001011010001111), \\
(0001011100010111), (0001011100011101), (0001011101000111), \\
(0001011110001101), (0001011110100011), (0001100011110101), \\
(0001101000111101), (0001101010001111), (0001110001110101), \\
(0001110100011101)
\end{tabular}
\smallskip
\caption{
$\Cyc(2,2,3)$:
For multiplicity $m=2$, alphabet size $q=2$, and word size $k=3$, there are 82 cyclic multi de Bruijn sequences.
We show the lexicographically least rotation of each.
}
\label{tab:(m,q,k)=(2,2,3)}
\end{table}

\section{Generating a uniform random multi de Bruijn sequence}
\label{sec:random_generation}

We present algorithms to select a uniform random multi de Bruijn sequence in the linear, linearized, and cyclic cases.

Kandel et al~\cite{Kandel1996}
present two algorithms to
generate random sequences with a specified number of occurrences of each $k$-mer:
(i)~a ``shuffling'' algorithm that permutes the characters of a sequence in a manner that preserves the number of times each $k$-mer occurs, and (ii)~a method to choose
a uniform random Eulerian cycle in an Eulerian graph
and spell out a cyclic sequence from the cycle.
Both methods may be applied to select linear or linearized multi de Bruijn sequences uniformly. However, additional steps are required to choose a uniform random cyclic de Bruijn sequence; see Sec.~\ref{sec:random_cyclic}.

\begin{example}
For $(m,q,k)=(2,2,2)$,
Table~\ref{tab:(m,q,k)=(2,2,2)}(A) lists the nine linearizations
starting with 00.
The random Eulerian cycle algorithm selects each
 with probability $1/9$.
Table~\ref{tab:(m,q,k)=(2,2,2)}(B) lists
the five cyclic de Bruijn sequences.
The four cyclic sequences of order 1
are each selected with probability $2/9$ since they have two linearizations
starting with $00$,
while
(00110011) has one such linearization and thus is selected with probability $1/9$.
Thus, the cyclic sequences are not chosen uniformly
(which is probability $1/5$ for each).
\end{example}

Kandel et al's 
shuffling algorithm performs a large number of random swaps of intervals in a sequence and
approximates the probabilities listed above.
Below, we focus on the random Eulerian cycle algorithm.

\subsection{Random linear(ized) multi de Bruijn sequences}
\label{sec:random_linear}
Kandel et al~\cite{Kandel1996} pre\-sent a general algorithm for finding random cycles in a directed Eulerian graph, and apply it to the de Bruijn graph of a string: the vertices are the $(k-1)$-mers of the string and the edges are the $k$-mers repeated with their multiplicities in the string.
As a special case, the de Bruijn graph of any cyclic sequence in $\Cyc(m,q,k)$ is the same as our $G(m,q,k)$.
We explain their algorithm and how it specializes to $G(m,q,k)$.

Let $G$ be a directed Eulerian multigraph
and $e=(v,w)$ be an edge in $G$.
 The proof of the BEST Theorem~\cite[Theorem~5b]{vanAE1951}
gives a bijection between
\begin{enumerate}
\item[(a)] Directed Eulerian cycles of $G$ whose first edge is $e$, and
\item[(b)] ordered pairs $(T,f)$ consisting of
 \begin{itemize}
 \item
 a spanning tree $T$ of $G$ rooted at $v$, and
 \item
  a function $f$ assigning each vertex $x$ of $G$ a 
permutation of its outgoing edges
except $e$ (for $x=v$) or the unique outgoing edge of $x$ in $T$ (for $x\ne v$).
 \end{itemize}
\end{enumerate}
In Eq.~(\ref{eq:BEST_product}), the left side counts (a) while the right side counts (b). Note that for each $T$, there are $\prod_{x\in V}(\outdeg(x)-1)!$ choices of $f$.

We will use the BEST Theorem map from (b) to (a).
Given $(T,f)$ and $e=(v,w)$ as above,
construct an Eulerian cycle $C$ with successive vertices $v_0,v_1,\ldots,v_\ell$
and successive edges $e_1,e_2,\ldots,e_\ell$ as follows:
\begin{itemize}
\item
Set $e_1=e$, $v_0=v$, $v_1=w$, and $i=1$.
\item
While $v_i$ has at least one outgoing edge not yet used in $e_1,\ldots,e_i$:
	\begin{itemize}
	\item
	Let $j$ be the number of times vertex $v_i$ appears among $v_0,\ldots,v_i$.
	\item
	If $j<\outdeg(v_i)$, set $e_{i+1}$ to the $j$\textsuperscript{th} edge
	 given by $f(v_i)$. \\
	If $j=\outdeg(v_i)$, set $e_{i+1}$ to the outgoing edge of $v_i$ in $T$.
	\item Let $v_{i+1}=\head(e_{i+1})$.
	\item Increment $i$.
	\end{itemize}
\end{itemize}

This reduces
the problem of selecting a uniform random Eulerian cycle
to selecting $T$ uniformly and $f$ uniformly. Kandel et al~\cite{Kandel1996} give the following algorithm to select a uniform random spanning tree $T$ 
with root $v$ in
$G$.

Initialize $T=\{v\}$.
Form a random walk in $G$ starting at $v_0=v$ and going backwards on edges (using incoming edges instead of outgoing edges):
$$
v_0 \xleftarrow{\;\,e_0}
v_1 \xleftarrow{\;\,e_1}
v_2 \xleftarrow{\;\,e_2}
v_3
\;\cdots
$$
At each vertex $v_i$, choose edge $e_i$ uniformly from its incoming edges,
and set $v_{i+1}=\tail(e_i)$.
(Kandel et al gives the option of ignoring loops in selecting an incoming edge, but we allow them.)
If this is the first time vertex $v_{i+1}$ has been encountered on the walk,
add vertex $v_{i+1}$ and edge $e_i$ to $T$.
Repeat until all vertices of $G$ have been visited at least once.
Kandel et al show that $T$ is
selected uniformly from spanning trees of $G$ directed towards root $v$.

We apply this to random spanning trees in $G(m,q,k)$.
Let $v\in\Omega^{k-1}$. Write $v=b_{k-2} b_{k-3} \cdots b_1 b_0$.
Select $b_i\in\Omega$ (for $i\ge k-1$) uniformly
and form a random walk with
$v_i=b_{i+k-2}\cdots b_i$ and $e_i$ labelled by $b_{i+k-1}\cdots b_i$:
$$
b_{k-2} \cdots b_1 b_0
\xleftarrow{b_{k-1}\cdots b_0}
b_{k-1} \cdots b_2 b_1
\xleftarrow{b_{k}\cdots b_1}
b_{k} \cdots b_3 b_2
\xleftarrow{b_{k+1}\cdots b_2}
\cdots
$$
For each $a\in\Omega$,
the fraction of incoming edges of $v_i$ with form
$ab_{i+k-2}\cdots b_i$
is $m/(mq)=1/q$,
so $b_{i+k-1}$
is
 selected uniformly from $\Omega$.
The first time vertex $v_{i+1}=b_{i+k-1}\ldots b_{i+1}$ is
encountered, add
 edge $e_i=b_{i+k-1}\ldots b_{i+1} b_i$
to $T$. Terminate as soon as all vertices have been encountered at least once.

Another way to generate a spanning tree $T$ of $G(m,q,k)$ is to generate a spanning tree $T'$ of $G(1,q,k)$ by the above algorithm, and then, for each edge in $T'$, uniformly choose one of the corresponding $m$ edges in $G(m,q,k)$. The probability of each spanning tree of $G(m,q,k)$ is the same as by the previous algorithm.
A similar observation is made in~\cite[Eq.~(6.2)]{vanAE1951}.
Since we
do not need to distinguish between the $m$ edges of $G(m,q,k)$ with the same $k$-mer,
it suffices to find a spanning tree of $G(1,q,k)$ rather than $G(m,q,k)$.

\begin{figure}
\includegraphics[width=.89\textwidth]{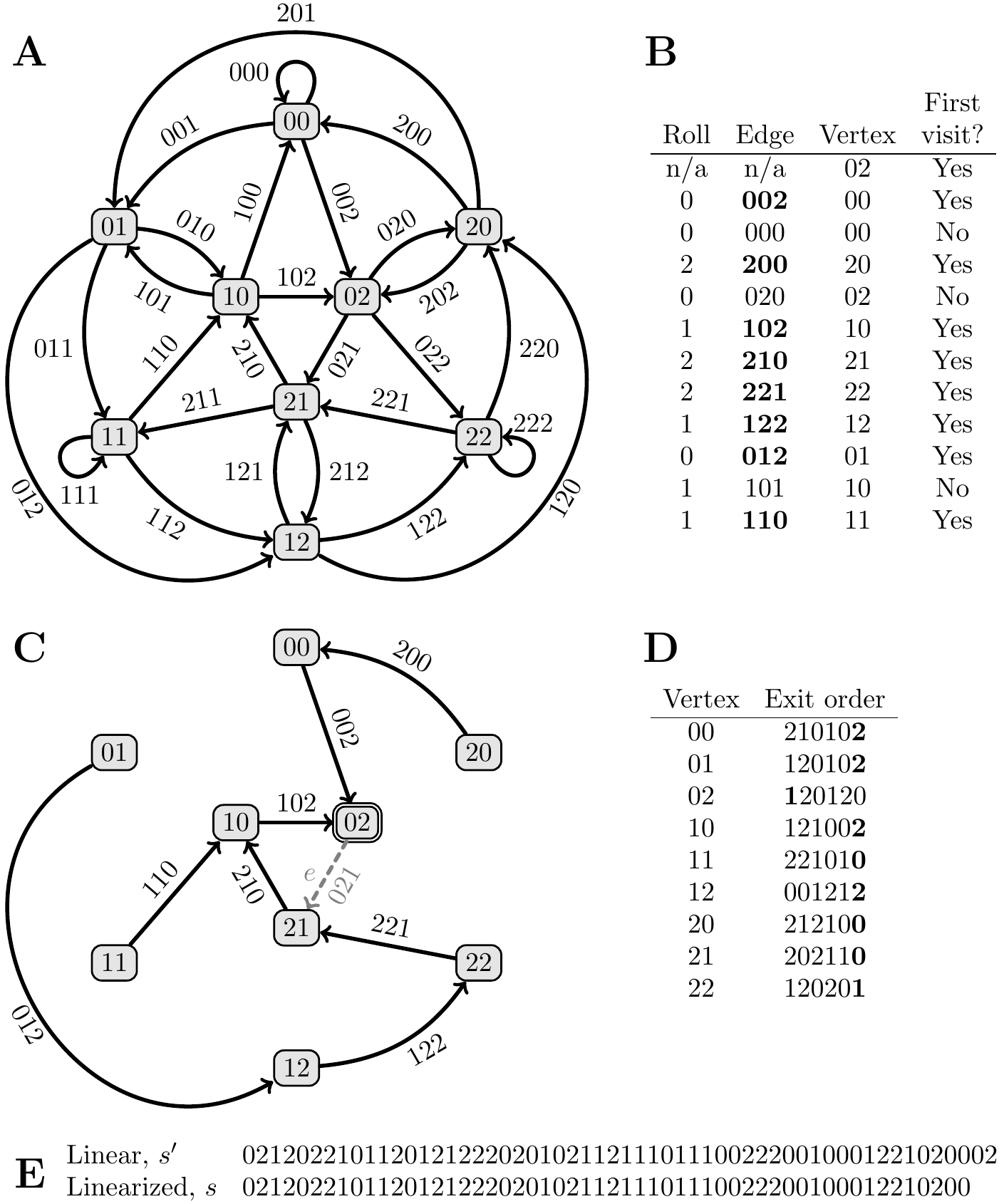}
\caption{Selecting a uniform random linear(ized) multi de Bruijn sequence
with
multiplicity $m=2$,
alphabet $\Omega=\{0,1,2\}$ of size $q=3$,
word size $k=3$, initial $k$-mer 021.
(A)~De Bruijn graph $G(1,q,k)$.
Edges are $k=3$-mers, vertices are $(k-1)=2$-mers.
(B)~Random process generates edges of a spanning tree
with root $v=02$.
(C)~Spanning tree (solid edges) with root $v=02$ (double border), plus
selected initial edge 021 (dashed).
(D)~Order to use outgoing edges at each vertex.
Each of the $q$ symbols occurs $m=2$ times.
Boldface in the last column encodes the tree; boldface in the first column
encodes initial edge; others are random.
(E)~Linear(ized) multi de Bruijn sequences
starting at $v=02$ and following edge orders (D).
}
\label{fig:rand_233}
\end{figure}

\begin{example}
\label{ex:random_tree}
Fig.~\ref{fig:rand_233}(A--C) illustrates finding a random spanning tree of $G(1,3,3)$ with root $v=02$.
Graph $G(1,3,3)$ is shown in Fig.~\ref{fig:rand_233}(A).
A fair 3-sided die gives rolls
$0,0,2,0,1,2,2,1,0,1,1,\ldots\,$,
shown in the first column of
Fig.~\ref{fig:rand_233}(B), leading to a random walk with edges and vertices indicated in the next columns.
On the first visit to each vertex (besides the root), its outgoing edge is boldfaced in panel (B) and added to the tree in panel (C).
\end{example}

Let $\codewords_{m,q}$ denote the set of all permutations of
$0^m 1^m \ldots (q-1)^m$; that is, sequences of length $mq$
with each symbol $0,\ldots,q-1$ appearing $m$ times.

For the graph $G(m,q,k)$,
we will replace
$(T,f)$ in the BEST Theorem bijection
by a function $g:\Omega^{k-1}\to\codewords_{m,q}$.
As we traverse a cycle $C$ starting on edge $e$, the successive visits to vertex $x\in\Omega^{k-1}$ have outgoing edges
with $k$-mers $xc_1,xc_2,\ldots,xc_{mq}$  (with $c_i\in\Omega$).
Set $g(x)=c_1 c_2 \ldots c_{mq}$.
For each $b\in\Omega$,
there are $m$ edges $xb$, so
exactly $m$ of the $c_i$'s equal $b$.
Thus, $g(x)\in\codewords_{m,q}$.
This encoding does not distinguish the $m$ edges $xb$ since permuting their order in the cycle does not change the sequence the cycle spells.
The first edge $e=(v,w)$ is encoded in the first position of $g(v)$;
the outgoing edge of $x$ in $T$ is
encoded in the last position of $g(x)$ for $x\ne v$;
and $f$ is encoded in the other $mq-1$ positions.

To pick a uniform random element of $\LC_y(m,q,k)$ or $\Lin(m,q,k)$:
\begin{itemize}
\item Input: $m,q,k$ (and $y$ for the linearized case).
\item In the linear case, select an initial $k$-mer $y$ uniformly from $\Omega^k$.
\item
Decompose $y=va$ ($v\in\Omega^{k-1}$, $a\in\Omega$).
\item
Select a uniform random spanning tree $T$ of $G(1,q,k)$ with root $v$.
\item
Each vertex
$x\in\Omega^{k-1}\setminus\{v\}$
has a unique outgoing edge in $T$, say $xb$ with $b\in\Omega$.
Set $g(x)$ to a random element of $\codewords_{m,q}$ ending in $b$.

\item Set $g(v)$ to a random element of $\codewords_{m,q}$ starting in $a$.

\item Traverse the cycle starting at $v$, following the edge order encoded in $g$, to form the sequence represented by this cycle.

\item
In the linearized case, delete the last $k-1$ characters of the sequence.

\end{itemize}

\begin{example}
Fig.~\ref{fig:rand_233} illustrates this algorithm for
$(m,q,k)=(2,3,3)$.
For a random element of $\Lin(m,q,k)$,
the initial $k$-mer is selected uniformly from $\Omega^k$; say $y=021$.
For a random element of $\LC_y(m,q,k)$, the $k$-mer $y$ is specified;
we specify $y=021$.
Panel (A) shows the de Bruijn graph $G(1,3,3)$.
Panels (B--C) select a random tree $T$ as described
in Example~\ref{ex:random_tree}.
Since $y=021$ (dashed edge in (C)), the root vertex is $v=02$; the $1$ in the last position
of $y$ is not used in selecting
$T$.

Panel (D) shows the order of exits from each vertex:
the function $g(x)=c_1\ldots c_{mq}$ gives that the $i$\textsuperscript{th} exit from $x$ is on edge $xc_i$.
Sequence $c_1\ldots c_{mq}$
is a permutation of 001122, with a constraint
(boldfaced in (D))
on either the first or last position and a random permutation on the other positions.
Since the linearization starts with 021, the first exit from vertex
$v=\,$02 is 1.
Tree $T$
is encoded in (D) as the last exit from
each vertex
except $v$;
e.g.,
since the tree has edge 002, the last exit from vertex 00 is 2.

To generate a 
sequence from (D),
start
 at $v=02$ with first exit 1
(edge 021).
At vertex 21, the first exit is 2 (edge 212).
At vertex 12, the first exit is 0 (edge 120).
At vertex 20, the first exit is 2 (edge 202).
Now we visit vertex 02 a second time. Its second exit is 2 (edge 022).
The sequence
so far is
0212022. Continue
until
reaching a vertex with no remaining exits;
this happens
on the $mq+1=7$\textsuperscript{th} visit to initial vertex $v=02$,
and yields a linear sequence (E).
The final $(k-1)$-mer, 02, duplicates the initial $(k-1)$-mer. Remove it from 
the end to obtain a linearization, also in (E).
\end{example}

\subsection{Random cyclic multi de Bruijn sequence}
\label{sec:random_cyclic}

Consider generating random cyclic sequences by
choosing a uniform random
element of $\Lin_y(m,q,k)$
and
circularizing it.
Each element of
$\Cyc^{(d)}(m,q,k)$ is represented by $m/d$ linearizations
and thus is generated with
probability $(m/d)/|\LC_y(m,q,k)|$,
which depends on $d$.
So this procedure does not select elements of $\Cyc(m,q,k)$ uniformly.
We will show how to adjust for $d$.

Partition $\LC_y(m,q,k)$ by rotation order, $d$:
$$
\LC_y(m,q,k) = \bigcup_{d|m} \LC^{(d)}_y(m,q,k) \;.
$$
We now construct certain subsets of this.
For each divisor $r$ of $m$, define
\begin{equation}
\LC_y(m,r;q,k) = \{ t^{m/r} : t \in \LC_y(r,q,k) \} \;.
\label{eq:define_LZ_y(m,r;q,k)}
\end{equation}
This set has size
$|\LC_y(m,r;q,k)|=|\LC_y(r,q,k)|$ and
this decomposition:
\begin{lemma}
\begin{equation}
\LC_y(m,r;q,k) =
\bigcup_{d: d|m \text{~and~} (m/r)|d} \LC^{(d)}_y(m,q,k) \;.
\label{eq:decompose_LZ_y(m,r;q,k)}
\end{equation}
\label{lem:decompose_LZ_y(m,r;q,k)}
\end{lemma}
\begin{proof}
Given $t\in\LC_y(r,q,k)$ of order $d'$,
then $t^{m/r}$ is in $\LC_y(m,q,k)$ and has order $d=(m/r)d'$.
Thus, the left side of~(\ref{eq:decompose_LZ_y(m,r;q,k)}) is a subset of the right side.
Conversely, consider $s\in\LC^{(d)}_y(m,q,k)$ and suppose
$(m/r)|d$. Then $s$ splits into $m/r$ equal segments: $s=t^{m/r}$.
Each $t$ is in $\LC_y(r,q,k)$ and has order $d'=d/(m/r)$.
So the right side of~(\ref{eq:decompose_LZ_y(m,r;q,k)}) is a subset of the left side.
\end{proof}

Now consider the following random process.
Fix $y\in\Omega^k$; it is sufficient to use $y=0^k$.
Let $p(r)$ be a probability distribution on the positive integer divisors of $m$. Generate a random element $\sigma$ of $\Cyc(m,q,k)$ as follows:
\begin{itemize}
\item Select a random divisor $r$ of $m$ with probability $p(r)$.
\item Select a uniform random sequence $t$ from $\LC_y(r,q,k)$.
\item Output $\sigma=(t^{m/r})$.
\end{itemize}

Consider all ways each
$\sigma\in\Cyc(m,q,k)$ can be generated by this algorithm.
Let $d$ be the order of $\sigma$.
Each $r|m$ is selected
with probability $p(r)$.
If $(m/r)|d$,
then by Lemma~\ref{lem:decompose_LZ_y(m,r;q,k)},
the $m/d$ linearizations of $\sigma$ beginning with $y$ are contained
in $\LC^{(d)}_y(m,r;q,k)$, and one of them may be generated
as $t^{m/r}$
with probability $(m/d)/|\LC_y(r,q,k)|$.
But if $m/r$ does not divide into $d$,
these linearizations
will not be generated.
In total,
\begin{equation}
P(\sigma) =
\sum_{\substack{r: \; r|m \\ \text{~and~} (m/d)|r}}
p(r) \cdot \frac{m/d}{|\LC_y(r,q,k)|} .
\label{eq:Cyc_d_probability}
\end{equation}
The following gives
$P(\sigma)=1/|\Cyc(m,q,k)|$
for all $\sigma\in\Cyc(m,q,k)$.
For $r|m$, set
\begin{align}
p(r) &= \frac{\phi(m/r)}{m} \cdot \frac{|\LC_y(r,q,k)|}{|\Cyc(m,q,k)|}
\label{eq:p(r)}
=
\frac{\phantom{\sum_{r|m}} \phi(m/r) \cdot (rq)!^{q^{k-1}} / r!^{q^k}}
	{\sum_{d|m}\phi(m/d) \cdot (dq)!^{q^{k-1}} / d!^{q^k}}
\end{align}
We used~(\ref{eq:S_y(m,q,k)}) and~(\ref{eq:Bcirc(m,q,k)})
to evaluate this.
To verify that $\sigma$ is selected uniformly, plug the
middle expression in~(\ref{eq:p(r)}) into~(\ref{eq:Cyc_d_probability}),
with $d$ equal to the order of $\sigma$:
\begin{align}
P(\sigma) &=
\sum_{\substack{r: \; r|m \\ \text{~and~} (m/r)|d}}
\frac{\phi(m/r)}{m} \cdot  \frac{|\LC_y(r,q,k)|}{|\Cyc(m,q,k)|}
\cdot
\frac{m/d}{|\LC_y(r,q,k)|}
\\
&=
\frac{1}{d\cdot |\Cyc(m,q,k)|}
\sum_{\substack{r: \; r|m \\ \text{~and~} (m/r)|d}}
\phi(m/r)
\label{eq:eval_P(s)}
\\
&=
\frac{1}{d\cdot |\Cyc(m,q,k)|} \cdot d
=
\frac{1}{|\Cyc(m,q,k)|} .
\end{align}
To evaluate the summation in~(\ref{eq:eval_P(s)}), substitute $u=m/r$.
Variable $u$ runs over divisors of $m$ that also divide $d$.
Since $d|m$, this is equivalent to $u|d$, so
\begin{align*}
\sum_{\substack{r: \; r|m \\ \text{~and~} (m/r)|d}}
\phi(m/r)
&=
\sum_{u|d} \phi(u) = d
\;.
\end{align*}

\section{Multicyclic de Bruijn Sequences and the Extended Burrows-Wheeler Transformation}
\label{sec:MCDB_EBWT}

\subsection{Multicyclic sequences}
Higgins~\cite{Higgins2012} defined a generalization of de Bruijn
sequences called \emph{de Bruijn sets} and showed how to generate them using an extension~\cite{Mantaci2005,Mantaci2007} of the Burrows-Wheeler Transformation (BWT)~\cite{BurrowsWheeler1994}. We
 generalize this
 to incorporate our multiplicity parameter $m$.

The length of a sequence $s$ is denoted $|s|$.
For $i\ge 0$,
let $s^i$ denote concatenating $i$ copies of $s$.
A sequence $s$ is \emph{primitive} if its length is positive and
$s$ is not a power of a shorter word.
Equivalently, a nonempty sequence $s$ is primitive iff the cycle $(s)$ is \emph{aperiodic},
that is, $|s|>0$ and the $|s|$ cyclic rotations of $s$ are distinct.
The \emph{root} of $s$ is the shortest prefix $t$ of $s$ such that
$s=t^d$ for some $d\ge 1$;
note that the root is primitive and
$d=|s|/|t|$ is the rotation order of $s$.
See~\cite[Sec.~3.1]{Mantaci2005},
\cite[Sec.~2.1]{Mantaci2007}, and~\cite[p.~129]{Higgins2012}.

For example, $abab$ is not primitive, but its root $ab$ is,
so $(ab)$ is an aperiodic cycle while $(abab)$ is not.
A \emph{multicyclic sequence} is a multiset of aperiodic cycles;
let $\Multicyc$ denote the set of all multicyclic sequences.
In $\Multicyc$,
instead of multiset notation
$\{(s_1),(s_2),\ldots,(s_r)\}$,
we will use cycle notation
$(s_1)(s_2)\ldots(s_r)$,
where different orderings of the cycles and different rotations are
considered to be equivalent;
this resembles permutation cycle notation, but symbols may appear multiple times and cycles are not composed as permutations.
For example, $\sigma=(a)(a)(ababbb)$
denotes a multiset $\{(a),(a),(ababbb)\}$
with two distinct cycles $(a)$ and one of $(ababbb)$.
This representation is not unique; e.g., $\sigma$ could also be
written as $(a)(bbabab)(a)$.

The \emph{length} of $\sigma=(s_1)\ldots(s_r)$ is $|\sigma|=\sum_{i=1}^r |s_i|$;
 e.g., $|(a)(a)(ababbb)|=1+1+6=8$. For $n\ge0$, let $\Multicyc_n = \{\sigma\in\Multicyc: |\sigma|=n\}$.

For $r\ge 0$,
the notation $(s)^r$
denotes $r$ copies of cycle $(s)$, while $\linstr{s}^r$
denotes concatenating
$r$ copies of $s$.
 E.g., $(ab)^3=(ab)(ab)(ab)$ while $\linstr{ab}^3=ababab$.
For $\sigma\in\Multicyc$, $\sigma^r$ denotes multiplying each cycle's multiplicity by $r$.

Let $s,w$ be nonempty linear sequences with $s$ primitive.
The number of occurrences of $w$
in $(s)$, denoted $\numoccur((s),w)$, is
$$|\{ j=0,1,\ldots,|s|-1 : \text{$w$ is a prefix of some power of $\rho^j(s)$} \}| \;.$$
For each $j$, it suffices to check if $s$ is a prefix of
$\linstr{\rho^j(s)}^{\CEIL{|s|/|w|}}$.
For example,
$(ab)$ has one occurrence of $babab$, because
$babab$ is a prefix of 
$\linstr{ba}^{\CEIL{5/2}}=\linstr{ba}^3=bababa$.
The number of occurrences of $w$ in $\sigma\in\Multicyc$ is
\begin{equation}
\numoccur(\sigma,w) = \sum_{(s)\in \sigma} \numoccur((s),w)
\end{equation}
where each cycle $(s)$ is repeated
with its multiplicity in $\sigma$.
For example, $bab$ occurs five times in $(ab)(ab)(ab)(baababa)$:
once in each $(ab)$ as a prefix of $\linstr{ba}^2=baba$, and twice in $(baababa)$ as prefixes of
 $bababaa$ and $babaaba$.

A \emph{multicyclic de Bruijn sequence} with parameters $(m,q,k)$ is
a multicyclic sequence
$\sigma\in\Multicyc$ over an alphabet $\Omega$ of size $q$,
such that every $k$-mer in $\Omega^k$ occurs exactly
$m$ times in $\sigma$.
Let $\MCDB(m,q,k)$ denote the set of such $\sigma$.
Higgins~\cite{Higgins2012} introduced the case $m=1$
and called it a \emph{de Bruijn set}.

\begin{example}
Each of
$aa,ab,ba,bb$ occurs twice in
$(a)(a)(ababbb)$.
$aa$ occurs once in the ``first'' $(a)$ as a prefix
of $a^2$, and once in the ``second'' $(a)$.
$ab,ba,bb$ each occur twice in $(ababbb)$,
including the occurrence of $ba$ that wraps around and the
two overlapping occurrences of $bb$.
This is a multicyclic de Bruijn sequence with
multiplicity $m=2$, alphabet size $q=2$, and word size $k=2$.
Table~\ref{tab:(m,q,k)=(2,2,2)}(D)
lists all multicyclic sequences with these parameters.
\end{example}

\begin{table}
\begin{tabular}{ll}
Definition & Notation \\ \hline
Linear sequence & $abc$ or $\linstr{abc}$	 \\
Cyclic sequence & $(abc)$ \\
Power of a linear sequence & $\linstr{abc}^r = abcabc\cdots$, $r$ times \\
Repeated cycle & $(abc)^r = (abc)(abc)\cdots$, $r$ times \\
Set of all multisets of aperiodic cycles & $\Multicyc$ \\
\quad(a.k.a.\ multicyclic sequences) & \\
Set of multicyclic sequences of length $n$	& $\Multicyc_n$ \\
Set of permutations of $0^m 1^m \cdots (q-1)^m$ & $\codewords_{m,q}$ \\
Burrows-Wheeler Transform & BWT \\
Extended Burrows-Wheeler Transform & EBWT \\
BWT table from forward transform & $\TB(s)$; $n\times n$ with $s\in\Omega^n$ \\
BWT table from inverse transform & $\TIB(w)$; $n\times n$ with $w\in\Omega^n$ \\
EBWT table from forward transform & $\TE(\sigma)$; $n\times c$ with $\sigma\in\Multicyc_n$ \\
EBWT table from inverse transform & $\TIE(w)$; $n\times c$ with $w\in\Omega^n$ \\
\end{tabular}

\smallskip

\caption{Notation for Extended Burrows-Wheeler Transformation}
\label{tab:notation_EBWT}

\end{table}

\subsection{Extended Burrows-Wheeler Transformation (EBWT)}
\label{sec:EBWT}

The Burrows-Wheeler Transformation~\cite{BurrowsWheeler1994}
is a certain map $\BWT:\Omega^n\to\Omega^n$ (for $n\ge 0$)
that preserves the number of times each character occurs.
For $n,q\ge 2$, it is neither injective nor surjective~\cite[p.~300]{Mantaci2007},
but it can be inverted
in practical applications
via
a ``terminator character'' (see Sec.~\ref{sec:BWT_vs_EBWT}).

Mantaci et al~\cite{Mantaci2005,Mantaci2007}
introduced the Extended Burrows-Wheeler Transformation, $\EBWT:\Multicyc_n\to\Omega^n$,
which modifies
the BWT to make it bijective.
As with the BWT, the EBWT also preserves the number
of times each character occurs.
EBWT is equivalent to a bijection
$\Omega^n\to\Multicyc_n$
of Gessel and Reutenauer~\cite[Lemma~3.4]{GesselReutenauer1993},
but computed by an algorithm similar to the BWT rather than
the method in~\cite{GesselReutenauer1993}.
There is also an earlier bijection
$\Omega^n\to\Multicyc_n$
by de Bruijn and Klarner~\cite[Sec.~4]{deBruijnKlarner1982}
based on factorization into 
Lyndon words~\cite[Lemma~1.6]{ChenFoxLyndon1958},
but it is not equivalent to the bijection in~\cite{GesselReutenauer1993}.

Higgins~\cite{Higgins2012} provided further analysis of
$\EBWT$ and used it to construct
de Bruijn sets.
We further generalize this to multicyclic de Bruijn sequences, as defined above.
We will give a brief description of the $\EBWT$ algorithm
and its inverse;
 see~\cite{Mantaci2005,Mantaci2007,Higgins2012} for proofs.

Let $\sigma=(s_1)\cdots(s_r)\in\Multicyc$.
Let $c$ be the least common multiple of
$|s_1|,\ldots,|s_r|$, and let $n=|\sigma|=\sum_{i=1}^r |s_i|$.

Consider a rotation
$t=\rho^j(s_i)$
of a cycle of $\sigma$.
Then $u=t^{c/|s_i|}$ has
 length $c$,
and $t$ can be recovered from $u$
as the root of $u$.

Construct a table $\TE(\sigma)$ with $n$ rows, $c$ columns,
and entries in $\Omega$:
\begin{itemize}
\item
Form a table of the powers of the rotations of each cycle of $\sigma$:
the $i$\textsuperscript{th} cycle ($i=1,\ldots,r$)
generates rows
$\linstr{\rho^j(s_i)}^{c/|s_i|}$ for $j=0,\ldots,|s_i-1|$.

\item
Sort the rows into lexicographic order to form $\TE(\sigma)$.
\end{itemize}
The \emph{Extended Burrows-Wheeler Transform of $\sigma$},
denoted $\EBWT(\sigma)$, is the linear sequence given by the last
column of $\TE(\sigma)$ read from top to bottom.

\begin{example}
\label{ex:EBWT_ex1}
Let $\sigma=(0001)(011)(1)$.
The number of columns is
$c=\lcm(4,3,1)=12$.
The rotations of each cycle are shown on the left, and the table $\TE(\sigma)$ obtained by sorting these is shown on the right:
$$
{\arraycolsep=3pt
\begin{array}{lcccccccccccc}
s_i & \multicolumn{12}{l}{\text{Powers of rotations of $s_i$}} \\ \hline
0001	& 0 & 0 & 0 & 1 & 0 & 0 & 0 & 1 & 0 & 0 & 0 & 1 \\
	& 1 & 0 & 0 & 0 & 1 & 0 & 0 & 0 & 1 & 0 & 0 & 0 \\
	& 0 & 1 & 0 & 0 & 0 & 1 & 0 & 0 & 0 & 1 & 0 & 0 \\
	& 0 & 0 & 1 & 0 & 0 & 0 & 1 & 0 & 0 & 0 & 1 & 0 \\ \hline
011	& 0 & 1 & 1 & 0 & 1 & 1 & 0 & 1 & 1 & 0 & 1 & 1 \\
	& 1 & 0 & 1 & 1 & 0 & 1 & 1 & 0 & 1 & 1 & 0 & 1 \\
	& 1 & 1 & 0 & 1 & 1 & 0 & 1 & 1 & 0 & 1 & 1 & 0 \\ \hline
1	& 1 & 1 & 1 & 1 & 1 & 1 & 1 & 1 & 1 & 1 & 1 & 1 \\
\end{array}}
\qquad
{\arraycolsep=3pt
\begin{array}{cccccccccccc}
 \multicolumn{12}{l}{\TE(\sigma)} \\ \hline
	 0 & 0 & 0 & 1 & 0 & 0 & 0 & 1 & 0 & 0 & 0 & 1 \\
	 0 & 0 & 1 & 0 & 0 & 0 & 1 & 0 & 0 & 0 & 1 & 0 \\
	 0 & 1 & 0 & 0 & 0 & 1 & 0 & 0 & 0 & 1 & 0 & 0 \\
	 0 & 1 & 1 & 0 & 1 & 1 & 0 & 1 & 1 & 0 & 1 & 1 \\
	 1 & 0 & 0 & 0 & 1 & 0 & 0 & 0 & 1 & 0 & 0 & 0 \\
	 1 & 0 & 1 & 1 & 0 & 1 & 1 & 0 & 1 & 1 & 0 & 1 \\
	 1 & 1 & 0 & 1 & 1 & 0 & 1 & 1 & 0 & 1 & 1 & 0 \\
	 1 & 1 & 1 & 1 & 1 & 1 & 1 & 1 & 1 & 1 & 1 & 1 \\
\end{array}
}
$$
The last column of
 $\TE(\sigma)$
 gives $\EBWT((0001)(011)(1))=10010101$.
\end{example}

\begin{example}
\label{ex:EBWT_ex2}
Let $\sigma=(0011)(0011)$.
Then $c=\lcm(4,4)=4$.
Since $(0011)$ has multiplicity 2,
each rotation of $0011$ generates two equal rows.
The last column of
 $\TE(\sigma)$
 gives $\EBWT((0011)(0011))=11001100$:
$$
{\arraycolsep=3pt
\begin{array}{lcccccc}
s_i & \\ \hline
0011	& 0 & 0 & 1 & 1 \\
	& 1 & 0 & 0 & 1 \\
	& 1 & 1 & 0 & 0 \\
	& 0 & 1 & 1 & 0 \\ \hline
0011	& 0 & 0 & 1 & 1 \\
	& 1 & 0 & 0 & 1 \\
	& 1 & 1 & 0 & 0 \\
	& 0 & 1 & 1 & 0 \\
\end{array}}
\qquad
{\arraycolsep=3pt
\begin{array}{cccc}
\multicolumn{4}{l}{\TE(\sigma)} \\ \hline
	 0 & 0 & 1 & 1 \\
	 0 & 0 & 1 & 1 \\
	 0 & 1 & 1 & 0 \\
	 0 & 1 & 1 & 0 \\
	 1 & 0 & 0 & 1 \\
	 1 & 0 & 0 & 1 \\
	 1 & 1 & 0 & 0 \\
	 1 & 1 & 0 & 0 \\
\end{array}}
$$
\end{example}

\subsection{Inverse of Extended Burrows-Wheeler Transformation}
\label{sec:invEBWT}
Without loss of generality, take $\Omega=\{0,1,\ldots,q-1\}$.
Let $w\in\Omega^n$.
Mantaci et al~\cite{Mantaci2005,Mantaci2007} give
the following method to
construct
$\sigma\in\Multicyc_n$
 with $\EBWT(\sigma)=w$.
Also see~\cite[Lemma~3.4]{GesselReutenauer1993} (which gives the same map but without relating it to the Burrows-Wheeler Transform) and~\cite[Theorem~1.2.11]{Higgins2012}.

Compute
the \emph{standard permutation} of $w=a_0 a_1 \ldots a_{n-1}$ as follows.
\begin{itemize}
\item
For each $i\in\Omega$, let $h_i$ be the number of times $i$ occurs in
$w$, and let the positions of the $i$'s be
$p_{i0} < p_{i1} < \cdots < p_{i,h_i-1} \;.$
\item
Form partial sums $H_i = \sum_{j<i} h_j$ (for $i\in\Omega$) and $H_q=\sum_{j\in\Omega} h_j=n$.
\item
The \emph{standard permutation} of $w$ is the permutation
$\pi$ on $\{0,1,\ldots,n-1\}$ defined by
$\pi(H_i+j) = p_{i,j}$ for $i\in\Omega$ and $j=0,\ldots,h_i-1$.
\end{itemize}

Compute
$\EBWT^{-1}:\Omega^*\to\Multicyc$,
 the \emph{inverse Extended Burrows-Wheeler Transform\/},
as follows:
\begin{itemize}
\item Compute $\pi(w)$ and
express it in permutation cycle form.
\item
In each cycle of $\pi(w)$, replace
entries
 in the range $[H_i,H_{i+1}-1]$ by
$i$.
\item
This forms an element $\sigma\in\Multicyc$. Output $\sigma$.
\end{itemize}
Mantaci et al~\cite{Mantaci2005}, \cite[Theorem~20]{Mantaci2007}
prove that these algorithms for
$\EBWT$ and $\EBWT^{-1}$
satisfy $\EBWT(\sigma)=w$ iff $\EBWT^{-1}(w)=\sigma$. Thus,
\begin{theorem}[Mantaci et al]
$EBWT:\Multicyc_n\to\Omega^n$ is a bijection.
\end{theorem}

\begin{example}
Let $w=a_0\ldots a_7 = 10010101$.
The four 0's have
positions 1, 2, 4, 6, so
$\pi(0)=1$, $\pi(1)=2$, $\pi(2)=4$, $\pi(3)=6$.
The four 1's have
positions 0, 3, 5, 7, so
$\pi(4)=0$, $\pi(5)=3$, $\pi(6)=5$, $\pi(7)=7$.
In cycle form, this permutation is
$\pi=(0,1,2,4)(3,6,5)(7)$.
In the cycle form,
replace $0,1,2,3$ by 0 and $4,5,6,7$ by 1 to obtain
$\sigma=(0,0,0,1)(0,1,1)(1)$.
Thus, $\EBWT^{-1}(10010101)=(0001)(011)(1)$.
\end{example}

The inverse EBWT may also be computed by constructing
the EBWT table $\TIE(w)$ of a linear sequence $w$.
This is analogous to the procedure used with the original BWT~\cite[Algorithm~D]{BurrowsWheeler1994},
and was made explicit for the EBWT in~\cite[pp.~130--131]{Higgins2012};
also see~\cite[pp.~182-183]{Mantaci2005} and~\cite[Theorem~20]{Mantaci2007}.
Computing the EBWT by this procedure is useful for
theoretical 
 analysis, but the standard permutation is more efficient for computational purposes.
\begin{itemize}
\item Input: $w\in\Omega^n$.
\item Form a table with $n$ empty rows.
\item Let $c=0$. This is the number of columns constructed so far.
\item Repeat the following until the last column equals $w$:
	\begin{itemize}
	\item Extend the table from $n\times c$ to $n\times(c+1)$ by
		shifting all $c$ columns one to the right and filling in the first column with $w$.
	\item Sort the rows lexicographically.
	\item Increment $c$.
	\end{itemize}
\item Output the $n\times c$ table $\TIE(w)$.
\end{itemize}

A \emph{Lyndon word}~\cite{Lyndon1954} is a linear sequence that is primitive and is lexicographically smaller than its nontrivial cyclic rotations.
To compute the inverse EBWT, form $\TIE(w)$ and take the primitive
root of each row. Output the multiset of cycles generated by
the primitive roots that are Lyndon words.
Note that if $w=\EBWT(\sigma)$, then the tables match,
$\TE(\sigma)=\TIE(w)$, and the inverse constructed this way satisfies
$\EBWT^{-1}(w)=\sigma$.
Mantaci et al~\cite{Mantaci2005,Mantaci2007} also consider
selecting other linearizations of the cycles,
rather than the Lyndon words or the cycles they generate.

\begin{example}
For $w=10010101$,
the table
$\TIE(10010101)$ is identical to the
righthand table
in Example~\ref{ex:EBWT_ex1}.
The primitive roots of the rows are
0001, 0010, 0100, 011, 1000, 101, 110, 1.
The Lyndon words among these are 0001, 011, and 1,
giving $\EBWT^{-1}(10010101)=(0001)(011)(1)$.
\end{example}

\begin{example}
For $w=11001100$, the table $\TIE(11001100)$ is identical to the
righthand table in Example~\ref{ex:EBWT_ex2}.
The primitive roots of the rows are
0011, 0011, 0110, 0110, 1001, 1001, 1100, 1100.
The Lyndon word among these is 0011 with multiplicity 2, giving
$\EBWT^{-1}(11001100)=(0011)(0011)$. 
\end{example}

\subsection{Applying the EBWT to multicyclic de Bruijn sequences}

\begin{table}
\begin{tabular}{ll}
$w$ & $\sigma$ \\ \hline
 0011 0011 & (0)(0)(01)(01)(1)(1) \\
 0011 0101 & (0)(0)(01)(011)(1) \\
 0011 0110 & (0)(0)(01)(0111) \\
 0011 1001 & (0)(0)(01011)(1) \\
 0011 1010 & (0)(0)(010111) \\
 0011 1100 & (0)(0)(011)(011) \\
 0101 0011 & (0)(001)(01)(1)(1) \\
 0101 0101 & (0)(001)(011)(1) \\
 0101 0110 & (0)(001)(0111) \\
 0101 1001 & (0)(001011)(1) \\
 0101 1010 & (0)(0010111) \\
 0101 1100 & (0)(0011)(011) \\
 0110 0011 & (0)(00101)(1)(1) \\
 0110 0101 & (0)(001101)(1) \\
 0110 0110 & (0)(0011101) \\
 0110 1001 & (0)(0011)(01)(1) \\
 0110 1010 & (0)(00111)(01) \\
 0110 1100 & (0)(0011011) \\
\end{tabular}
\qquad
\begin{tabular}{ll}
 $w$ & $\sigma$ \\ \hline
 1001 0011 & (0001)(01)(1)(1) \\
 1001 0101 & (0001)(011)(1) \\
 1001 0110 & (0001)(0111) \\
 1001 1001 & (0001011)(1) \\
 1001 1010\textsuperscript{*} & (00010111)\textsuperscript{*} \\
 1001 1100 & (00011)(011) \\
 1010 0011 & (000101)(1)(1) \\
 1010 0101 & (0001101)(1) \\
 1010 0110\textsuperscript{*} & (00011101)\textsuperscript{*} \\
 1010 1001 & (00011)(01)(1) \\
 1010 1010 & (000111)(01) \\
 1010 1100\textsuperscript{*} & (00011011)\textsuperscript{*} \\
 1100 0011 & (001)(001)(1)(1) \\
 1100 0101 & (001)(0011)(1) \\
 1100 0110 & (001)(00111) \\
 1100 1001 & (0010011)(1) \\
 1100 1010\textsuperscript{*} & (00100111)\textsuperscript{*} \\
 1100 1100\textsuperscript{*} & (0011)(0011)\textsuperscript{*} \\
\end{tabular}

\smallskip

\caption{$\MCDB(2,2,2)$: Multicyclic de Bruijn sequences $\sigma$ for 
$m=q=k=2$,
and their Extended Burrows-Wheeler Transforms $w=\EBWT(\sigma)$.
Entries corresponding to cyclic multi de Bruijn sequences $\Cyc(2,2,2)$ are marked `*'.}
\label{tab:mcdb(2,2,2)}
\end{table}

We have the following generalization of Higgins~\cite[Theorem~3.4]{Higgins2012}, which was for the case $m=1$. Our proof is similar to the one in~\cite{Higgins2012} but introduces the multiplicity $m$.
Table~\ref{tab:mcdb(2,2,2)} illustrates the theorem for
$(m,q,k)=(2,2,2)$.

\begin{theorem}
For $m,q,k\ge 1$,

(a)~The image of $\MCDB(m,q,k)$ under $\EBWT$ is $(\codewords_{m,q})^{q^{k-1}}$.

(b)~$
|\MCDB(m,q,k)|=
W(m,q,k) = \frac{(mq)!^{q^{k-1}}}{m!^{q^k}}
\;.
$
\label{thm:MCDB_bijection}
\end{theorem}

\begin{proof}
When $q=1$, the theorem holds:
\begin{align*}
\MCDB(m,1,k)&=\{(0)^m\} \text{~has size 1}, \\
(\codewords_{m,q})^{q^{k-1}}=\codewords_{m,1}&=\{0^m\}
\phantom{()}\text{~has size $W(m,1,k)=1$}, \\
\EBWT((0)^m)&=0^m.
\end{align*}
Below, we prove (a) for $q\ge 2$.
Since $\EBWT$ is invertible,
(b) follows from
$|\MCDB(m,q,k)|=|\EBWT(\MCDB(m,q,k))|=|(\codewords_{m,q})^{q^{k-1}}|=W(m,q,k)$.

In both the forwards and inverse EBWT table, the rows starting with nonempty sequence $x$ (meaning $x$ is a prefix of a suitable power of the row) are consecutive (since the table is sorted), yielding a block $B_x$ of consecutive rows.
In the forwards direction,
in $\TE(\sigma)$,
the number of rows in $B_x$ equals
$\numoccur(\sigma,x)$,
since for each occurrence of $x$ in $\sigma$, a row is generated
by rotating the occurrence to the beginning
and taking a suitable power.

\smallskip

\noindent
\emph{Forwards:}
Given $\sigma\in\MCDB(m,q,k)$, we show that $\EBWT(\sigma)\in(\codewords_{m,q})^{q^{k-1}}$.

Every $k$-mer occurs exactly $m$ times in $\sigma$,
so
the rows of $\TE(\sigma)$ are partitioned
by their initial $k$-mer
 into
$q^k$ blocks of $m$ consecutive rows.
When $q>1$, the table has at least $k$ columns since $01^{k-1}$
is in $\sigma$ and generates a row of size at least $k$.

For each $(k-1)$-mer $x\in\Omega^{k-1}$,
block $B_x$ of $\TE(\sigma)$ has $mq$ rows,
because $\sigma$ has $m$ occurrences of
$k$-mer $xb$ for each of the $q$ choices of $b\in\Omega$.

Suppose by way of contradiction that at least $m+1$ rows of $B_x$ end in the same symbol $a\in\Omega$.
Upon cyclically rotating the table one column to the right
and sorting the rows (which leaves the table invariant),
at least $m+1$ rows begin with $k$-mer $ax$,
which contradicts that exactly $m$ rows start with
each $k$-mer.
Thus, at most $m$ rows of $B_x$
 end in the same symbol.
 Since $B_x$ has $mq$ rows and each of the $q$ symbols occurs
 at the end of at most $m$ of them, in fact, each symbol must occur
 at the end of exactly $m$ of them.
 Let $w_x$ be the word formed by the last character in the rows of $B_x$,
 from top to bottom. Then $w_x\in\codewords_{m,q}$.
 $\EBWT(\sigma)$ is the last column of the table, which is the
 concatenation of the $q^{k-1}$ words $w_x$ for $x\in\Omega^{k-1}$
 in order by $x$.

\smallskip

\noindent
\emph{Inverse:}
Given $w \in (\codewords_{m,q})^{q^{k-1}}$,
we show that $\EBWT^{-1}(w)\in\MCDB(m,q,k)$.

For $j\in\{1,\ldots,k\}$, we use induction to show that exactly
$mq^{k-j}$ rows of
 $\TIE(w)$
start with each
 $x\in\Omega^j$.
See
Examples~\ref{ex:EBWT_ex1}--\ref{ex:EBWT_ex2},
in which
$mq^{k-1}=2\cdot 2^{2-1}=4$ rows start with each of 0 and 1
and $mq^{k-2}=2\cdot 2^{2-2}=2$ rows start with each of 00, 01, 10, and 11.

\emph{Base case $j=1$:}
Each letter in $\Omega$
occurs exactly $mq^{k-1}$ times in $w$.
Upon cyclically shifting the last column, $w$, of the table
to the beginning
and sorting rows to obtain the first column,
this gives exactly $mq^{k-1}$ occurrences
of each letter in the first column. This proves the case $j=1$.

\emph{Induction step:} Given $j\in\{2,\ldots,k\}$, suppose the statement holds for case $j-1$.
For each $x\in\Omega^{j-1}$,
block $B_x$ consists of
$mq^{k-j+1}$ consecutive rows.
Split $B_x$ into into $q^{k-j}$ subblocks of $mq$ consecutive rows
(we may do this since $j\le k$).
The last column of each subblock
forms an interval in $w$ of length $mq$ taken from $\codewords_{m,q}$,
so each $a\in\Omega$ occurs at the end of exactly $m$ rows in
each subblock. Thus, in the whole table,
exactly $m \cdot q^{k-j}$ rows
start in $(k-1)$-mer $x$ and end
in character $a$.
On rotating the last column to the beginning,
exactly $mq^{k-j}$ rows start with $k$-mer $ax$.

Case
$j=k$ shows that every $k$-mer in $\Omega^k$
is a prefix of exactly $m$ rows of $\TIE(w)$.
Thus, every $k$-mer occurs exactly $m$ times in $\EBWT^{-1}(w)$,
so $\EBWT^{-1}(w)\in\MCDB(m,q,k)$.
\end{proof}

Theorem~\ref{thm:MCDB_bijection}(a) may be used to generate multicyclic de Bruijn sequences:
\begin{itemize}
\item
$\MCDB(m,q,k)=\{\EBWT^{-1}(w):w\in\codewords_{m,q}\}$;
see Table~\ref{tab:mcdb(2,2,2)}.
\item
To select a uniform random element of $\MCDB(m,q,k)$,
do $n=q^{k-1}$ rolls of a fair $|\codewords_{m,q}|=(mq)!/m!^q$-sided die to pick
$w_0,\ldots,w_{n-1}\in\codewords_{m,q}$.
Compute $\EBWT^{-1}(w_0\ldots w_{n-1})$.
\end{itemize}

\subsection{Original vs. Extended Burrows-Wheeler Transformation}
\label{sec:BWT_vs_EBWT}

To compute the original Burrows-Wheeler Transform $\BWT(s)$ of $s\in\Omega^n$ (see~\cite{BurrowsWheeler1994}):
form an $n\times n$ table, $\TB(s)$, whose rows are
$\rho^0(s),\ldots,\rho^{n-1}(s)$ sorted into lexicographic order.
Read
the last column from top to bottom to form a linear sequence, $\BWT(s)$.
This is similar to computing $\EBWT((s))$, 
but $s$ does not have to be primitive,
and the table is always square.

Given $w\in\Omega^n$, we construct the inverse BWT table $\TIB(w)$ by the same algorithm as for $\TIE(w)$, except we stop at $n$ columns rather than when the last column equals $w$.
Table $\TIB(w)$ is defined for all $w$,
but $\BWT^{-1}(w)$ may or may not exist. If $\TIB(w)$ consists of the $n$ rotations of its first row, sorted into order, then
one or more inverses $\BWT^{-1}(w)$ exist (take any row of the table as the inverse); otherwise, an inverse
$\BWT^{-1}(w)$
does not exist.

Both $s$ and its rotations have the same BWT table and the same transforms
(see~\cite[Prop.~1]{Mantaci2007}): $\TB(\rho^j(s))=\TB(s)$ and $\BWT(\rho^j(s))=\BWT(s)$
for $j=0,\ldots,|s|-1$.
Given $w=\BWT(s)$, the inverse may only be computed up to an unknown rotation, $\rho^j(s)$.
As an aside, we note the following, but we do not use it:
in order to recover a linear sequence $s$ from the inverse,
Burrows and Wheeler~\cite[Sec.~4.1]{BurrowsWheeler1994}
introduce a 
terminator character (which we denote `\$') that
does not occur in the input.
To encode, compute $w=\BWT(s\$)$.
To decode, compute $\BWT^{-1}(w)$,
select the rotation that puts `\$' at the end,
and delete `\$' to recover $s$.
Terminator characters are common in implementations of the BWT, but
since our application is cyclic sequences rather than linear sequences, we do not use a terminator.

If $s$ is primitive, then
$\TB(s)=\TE((s))$ and $\BWT(s)=\EBWT((s))$.
More generally,
any sequence $s$ can be decomposed as $s=t^d$ with $t$ primitive and $d\ge 0$, for which we have:

\begin{theorem}
Let $s=t^d$,
with $t$
 primitive and $d\ge 0$.
Then $\BWT(s)=\EBWT((t)^d)={a_1}^d {a_2}^d \cdots {a_r}^d$, where $r=|t|$ and $\BWT(t)=\EBWT((t))=a_1\ldots a_r$.
\label{thm:BWT(t^d)}
\end{theorem}

\begin{theorem}
\strut

(a)~Given $\EBWT(\sigma)=a_1 \ldots a_r$, then
$\EBWT(\sigma^d)= {a_1}^d \ldots {a_r}^d$.

(b)~Given $\EBWT^{-1}(a_1\ldots a_r)=\sigma$,
then $\EBWT^{-1}({a_1}^d \ldots {a_r}^d) = \sigma^d$.
\label{thm:EBWT(sigma^d)}
\end{theorem}

In the following, ``$\BWT^{-1}(w)$ exists'' means that at least one inverse exists, not that a unique inverse exists.
\begin{theorem}
(a)~$\BWT^{-1}(w)$ exists iff $\EBWT^{-1}(w)$ has the form
$(t)^d$, where $t$ is primitive and $d\ge 0$ (specifically $d=|w|/|t|$).
In this case, $\BWT^{-1}(w)=t^d$ (or any rotation of it).

(b)~$s_1=\BWT^{-1}(a_1a_2\cdots a_r)$ exists
iff
$s_2=\BWT^{-1}({a_1}^d{a_2}^d\cdots{a_r}^d)$
exists. If these exist, then up to a rotation, $s_2={s_1}^d$.
\label{thm:BWT^(-1)_existence}
\end{theorem}
Theorem~\ref{thm:BWT^(-1)_existence}(b)
is essentially equivalent to~\cite[Prop.~2]{Mantaci2007}.

The above three theorems
are trivial when the input is null.
In the proofs below, we assume the input sequences have positive length.

\begin{proof}[Proof of Theorem~\ref{thm:BWT(t^d)}]
Let $n=|s|=rd$.
There are $|t|=r$ distinct rotations of $s$.
In the sequence $\rho^j(s)$ for $j=0,\ldots,n-1$, each distinct rotation is generated $d$ times.
These are sorted into order to form $\TB(s)$, giving
$r$ blocks of $d$ consecutive equal rows.
Therefore, the last column of the table has the form
${a_1}^d {a_2}^d \cdots {a_r}^d$.
The BWT construction ensures this is a permutation of the input 
$s=t^d$, so $a_1\ldots a_r$ is a permutation of $t$.

Let $\sigma=(t)^d$.
Since $t$ is primitive, $\TE(\sigma)$ consists
of each rotation of $t$ repeated on $d$ consecutive rows, sorted into order.
This is identical to the first $r$ columns of $\TB(s)$.
Conversely,
$\TB(s) = [ \TE(\sigma) \; \cdots \; \TE(\sigma) ]$ ($d$ copies of $\TE(\sigma)$ horizontally).
Thus, the last columns of $\TB(s)$ and $\TE(\sigma)$ are the same,
so $\BWT(s)=\EBWT(\sigma)$.
\end{proof}

\begin{proof}[Proof of Theorem~\ref{thm:EBWT(sigma^d)}]
(a) Replicate each row of $\TE(\sigma)$ on $d$ consecutive rows to
form $\TE(\sigma^d)$.
Given that the last column of $\TE(\sigma)$ is $\EBWT(\sigma)=a_1\ldots a_r$,
then the last column of $\TE(\sigma^d)$ is $\EBWT(\sigma^d)={a_1}^d \ldots {a_r}^d$.

(b) Applying $\EBWT^{-1}$ to both equations in (a) gives (b).
\end{proof}

\begin{proof}[Proof of Theorem~\ref{thm:BWT^(-1)_existence}]

(a)~Suppose $s=\BWT^{-1}(w)$ exists.
Decompose $s=t^d$ where $t$ is primitive and $d\ge 0$.
Then
$w=\BWT(t^d)$,
which, by Theorem~\ref{thm:BWT(t^d)},
equals $\EBWT((t)^d)$, so $\EBWT^{-1}(w)=(t)^d$.

Conversely, suppose $\EBWT^{-1}(w)=(t)^d$,
where $t$ is primitive.
By Theorem~\ref{thm:BWT(t^d)},
$\BWT(t^d)=\EBWT((t)^d)=w$,
so $\BWT^{-1}(w)=t^d$.

Note that $\BWT^{-1}(w)$
and linearizations of cycles of $\EBWT^{-1}(w)$
are only defined up to a rotation,
so other rotations of $s$ and $t$ may be used.

(b) Let $w=a_1\ldots a_r$ and $w'={a_1}^d \ldots {a_r}^d$.
Let $\sigma=\EBWT^{-1}(w)$. By Theorem~\ref{thm:EBWT(sigma^d)}, $\EBWT^{-1}(w') = \sigma^d$.
If $\sigma$ has the form $(t)^h$ for some primitive cycle $t$
and integer $h\ge 0$, then
$\sigma^d=(t)^{hd}$.
By part (a),
$\BWT^{-1}(w)=t^h$ and $\BWT^{-1}(w')=t^{hd}$
(or any rotations of these),
so $\BWT^{-1}(w')$ is the $d$\textsuperscript{th} power
of $\BWT^{-1}(w)$.
But if $\sigma$ doesn't have that form, then
$\BWT^{-1}(w)$ and $\BWT^{-1}(w')$ don't exist.
\end{proof}

Set
$
\nod = \{ w \in \Omega^+ : \text{$w$ does not have the form ${a_1}^i{a_2}^i\cdots$ for any $i>1$} \} .
$

\begin{theorem}
Let $(s)\in\Cyc^{(d)}(m,q,k)$, with $m,q,k\ge 1$.
Then $\BWT(s)$ has the form ${a_1}^d\ldots{a_r}^d$
where $a_1\ldots a_r\in (\codewords_{m/d,q})^{q^{k-1}}\cap\nod$.
\end{theorem}
\begin{proof}
Since $s$ has order $d$, it has the form $s=t^d$ with $t$ primitive.
By Lemma~\ref{lem:m/d},  $t\in\Cyc^{(1)}(m/d,q,k)$.
Since $t$ is primitive, $\BWT(t)=\EBWT((t))$; say this is $u=a_1\ldots a_r$.
By Theorem~\ref{thm:MCDB_bijection}(a),
$u\in(\codewords_{m/d,q})^{q^{k-1}}$,
and by Theorem~\ref{thm:BWT(t^d)},
$\BWT(s)={a_1}^d\ldots {a_r}^d$

Now suppose by way of contradiction that
$u = {b_1}^i \ldots {b_{r/i}}^i$
with $i|r$ and $i>1$.
By Theorem~\ref{thm:EBWT(sigma^d)}(b),
$\EBWT^{-1}({b_1}^i \ldots {b_{r/i}}^i)$
has at
least $i$ cycles;
but $\EBWT^{-1}(u)=(t)$
has exactly one cycle,
 a contradiction.
Thus,
$u\in\nod$.
\end{proof}

\begin{example}
\label{ex:bwt(cyc(2,2,2))}
Cyclic multi de Bruijn sequences in $\Cyc(2,2,2)$
correspond
to entries marked `*' in Table~\ref{tab:mcdb(2,2,2)}.
Spaces highlight the structure from
Theorem~\ref{thm:MCDB_bijection}(a) but are not part of the sequences.
$$\arraycolsep=2pt\begin{array}{lll}
\BWT(00010111)&=\EBWT((00010111))&=1001\;1010
\\
\BWT(00011011)&=\EBWT((00011011))&=1010\;1100
\\
\BWT(00011101)&=\EBWT((00011101))&=1010\;0110
\\
\BWT(00100111)&=\EBWT((00100111))&=1100\;1010
\\
\BWT(00110011)&=\EBWT((0011)(0011))&=1100\;1100
= 1^20^21^20^2
\;.
\end{array}
$$
By Theorem~\ref{thm:BWT(t^d)},
$\BWT(00110011)$ is derived from
this example with $m=1$:
$$
\BWT(0011)=\EBWT((0011))=10 \; 10
\;.
$$
\end{example}

\section{Partitioning a graph into aperiodic cycles}
\label{sec:MCDB_graph_cycles}

We give an alternate proof of Theorem~\ref{thm:MCDB_bijection}(b) based on the graph $G(m,q,k)$ rather than the Extended Burrows-Wheeler Transform.

Let $G$ be a multigraph with vertices $V(G)$ and edges $E(G)$.

Denote the sets of incoming and outgoing edges at vertex $x$ by
$\In_G(x)=\{e\in E(G): \head(e)=x\}$
and
$\Out_G(x)=\{e\in E(G): \tail(e)=x\}$.

The \emph{period} of
a cycle on edges 
 $(e_1,\ldots,e_n)$
is the smallest $p\in\{1,\ldots,n\}$ for which $e_i=e_{i+p}$
(subscripts taken mod $n$) for $i=1,\ldots,n$.
A cycle is \emph{aperiodic} when its period is its length.

Let $\vec \nu = \langle \nu_e : e\in E(G)\rangle$
be an assignment of nonnegative integers
to edges of $G$.
Let $\Multicyc_G(\vec \nu)$ be the collection of multisets of
aperiodic cycles in $G$ in which each edge
$e$ is used exactly $\nu_e$ times in the multiset.
Different representations of a cycle, such as $(e_1,\ldots,e_n)$
vs. $(e_{i+1},\ldots,e_n,e_1,\ldots,e_i)$, are
 considered equivalent.
 In this section,
we will
determine $|\Multicyc_G(\vec\nu)|$.
Set
\begin{align*}
\indeg_G(x,\nu)&=\sum_{e\in\In_G(x)} \nu_e
&
\outdeg_G(x,\nu)&=\sum_{e\in\Out_G(x)} \nu_e
\;.
\end{align*}
A necessary condition for such a multiset to exist is that
at each vertex $x\in V(G)$,
$\indeg_G(x,\nu)=\outdeg_G(x,\nu)$,
because the number of times vertex $x$ is entered (resp., exited)
among the cycles is given by
$\indeg_G(x,\nu)$ (resp., $\outdeg_G(x,\nu)$).
 Below, we assume this holds.
 
\begin{figure}
\includegraphics[width=.75\textwidth]{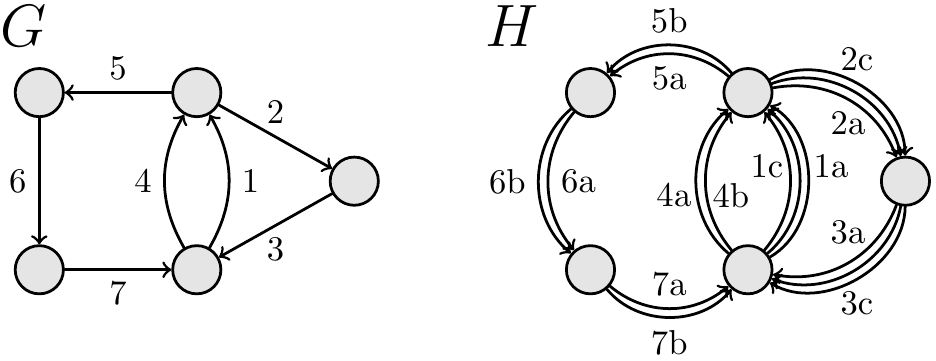}
\caption{Given directed multigraph $G$ and
edge multiplicities $\nu_1=\nu_2=\nu_3=3$, $\nu_4=\nu_5=\nu_6=\nu_7=2$,
replace each edge $e$ of $G$ by $\nu_e$ edges to form $H$.
For example, edge $2$ in $G$ yields edges $2a,2b,2c$ in $H$.}
\label{fig:G->H}
\end{figure}

Form a directed multigraph
$H$
 with the same vertices as $G$,
and with each edge $e$ of $G$ replaced by $\nu_e$ distinguishable
 edges on the same vertices as $e$.
See Fig.~\ref{fig:G->H}.
Since $G$ is a multigraph, if there are $r$ edges $e_1,\ldots,e_r$ from
$v$ to $w$ in $G$, there will be $\sum_{i=1}^r \nu_{e_i}$ edges from $v$ to $w$ in $H$.
By construction, at every vertex $x$,
$\indeg_H(x)=\indeg_G(x,\nu)$ and $\outdeg_H(x)=\outdeg_G(x,\nu)$.
Combining this
with $\indeg_G(x,\nu)=\outdeg_G(x,\nu)$
gives that $H$ is balanced
(at each vertex, the sum of the indegrees equals the sum of the outdegrees).
Note that we do not require $G$ or $H$ to be connected.

Define a map $\pi:H\to G$ that preserves vertices and maps the $\nu_e$ edges of $H$ corresponding to $e$ back to $e$ in $G$.
Extend $\pi$ to map multisets of cycles in $H$ to multisets of cycles in $G$
as follows:
\begin{multline*}
\pi\bigl((e_{1,1},\ldots,e_{1,n_1})
\cdots (e_{r,1},\ldots,e_{r,n_r})\bigr)
\\
= \bigl(\pi(e_{1,1}),\ldots,\pi(e_{1,n_1})\bigr)
\cdots \bigl(\pi(e_{r,1}),\ldots,\pi(e_{r,n_r})\bigr) \;.
\end{multline*}

A \emph{cycle partition} of $H$ is a set of cycles in $H$,
with each edge of $H$ used
once
over the set.
Let $\CP(H)$ denote the set of all cycle partitions of $H$.

An \emph{edge successor map} of $H$ is a function $f:E(H)\to E(H)$ such that for every edge $e\in E(H)$, $\head(e)=\tail(f(e))$, and
for every vertex $x\in V(H)$, $f$ restricts to a bijection
$f:\In_H(x)\to\Out_H(x)$.
Let $F(H)$ be the set of all edge successor maps of $H$.

\begin{theorem}
There is a bijection between $\CP(H)$ and $F(H)$.
\label{thm:cyc_partition_edge_successor}
\end{theorem}
Denote the edge successor map corresponding to cycle partition $C$ by $f_C$, and the cycle partition corresponding to edge successor map $f$ by $C_f$.

\begin{proof}
Let $C\in \CP(H)$.
Construct
$f_C\in F(H)$
 as follows:
for each edge $e\in E(H)$, set $f_C(e)=e'$ where $e'$ is
the unique edge following $e$ in its cycle in $C$;
note that $\head(e)=\tail(e')$, as required.
Every edge appears exactly once in $C$
and
has exactly one image and one inverse in $f_C$.
Thus, at each vertex $x$, $f_C$ restricts to a bijection
$f_C:\In_H(x)\to\Out_H(x)$,
so $f_C\in F(H)$.

Conversely, given $f\in F(H)$, construct a cycle partition $C_f$ as follows.
$f$ is a permutation of the finite set $E(H)$.
Express this permutation in cycle form, $C_f$.
Each permutation cycle has the form $(e_1,e_2,\ldots,e_n)$
with $f(e_i)=e_{i+1}$ (subscripts taken mod $n$).
Since $\head(e_i)=\tail(f(e_i))=\tail(e_{i+1})$, these
permutation cycles are also graph cycles,
so $C_f\in\CP(H)$.

By construction, the maps $f\mapsto C_f$ and $C\mapsto f_C$ are inverses.
\end{proof}

\begin{corollary}
Let $H$ be a finite balanced directed multigraph, with all edges distinguishable. The number of cycle partitions of $H$ is
\begin{equation}
|\CP(H)|=
\prod_{x\in V(H)} (\outdeg(x))!
\label{eq:num_cycle_partitions}
\end{equation}
\label{cor:num_cycle_partitions}
\end{corollary}

\begin{proof}
At each $x\in V(H)$, there are $(\outdeg(x))!$
bijections
from $\In(x)$ to $\Out(x)$,
so $|F(H)|$
is given by~(\ref{eq:num_cycle_partitions}).
By
Theorem~\ref{thm:cyc_partition_edge_successor},
the number of cycle partitions of $H$ is also given
by~(\ref{eq:num_cycle_partitions}).
\end{proof}

Cycles in $\pi(C_f)$ are not necessarily aperiodic, but
may be split into aperiodic cycles as follows.
Replace each cycle $D=(e_1,\ldots,e_n)$ of $\pi(C_f)$
by $n/p$ cycles $(e_1,\ldots,e_{n/p})$, where $p$ is the period of $D$.
This is well-defined; had we represented $D$
by a different rotation, the result would be equivalent,
but represented by a different rotation.
Let $\Split(\pi(C_f))$ denote splitting all cycles of $\pi(C_f)$ in this fashion;
this is a multiset of aperiodic cycles in $G$.

\begin{theorem}
If $\indeg_G(x,\nu)=\outdeg_G(x,\nu)$ at every
 $x\in V(G)$, then
\begin{equation}
|\Multicyc_G(\vec\nu)| =
\frac{\prod_{x\in V(G)} (\outdeg_G(x,\nu))!}{\prod_{e\in E(G)} \nu_e!}
.
\label{eq:|M_G(nu)|}
\end{equation}
\end{theorem}
\begin{proof}
By Corollary~\ref{cor:num_cycle_partitions}, the number of cycle partitions of $H$ is
\begin{equation}
\prod_{x\in V(H)} (\outdeg_H(x))!
=
\prod_{x\in V(G)} (\outdeg_G(x,\nu))!
\label{eq:num_cyc_partitions_Gnu}
\end{equation}
In $F(H)$, define $f \equiv f\,'$ iff
every edge $e\in V(H)$ satisfies $\pi(f(e))=\pi(f'(e))$.
Although $C_f$ and $C_{f'}$ may have different cycle structures,
on splitting the cycles,
$\Split(\pi(C_f))=\Split(\pi(C_{f'}))$.
The size of each equivalence class is
$\prod_{e\in E(G)} \nu_e!$.
Dividing~(\ref{eq:num_cyc_partitions_Gnu}) by this gives~(\ref{eq:|M_G(nu)|}).
\end{proof}

\begin{example}
In Fig.~\ref{fig:G->H}, these
have different cycle structures:
\begin{align*}
C_f &= (4a,5a,6a,7a)(1a,2a,3a,1b,2b,3b)(1c,5b,6b,7b,4b,2c,3c) \\
C_{f'} &= (4a,5a,6a,7a)(1a,2a,3a)(1b,2b,3b)(1c,5b,6b,7b,4b,2c,3c)
\end{align*}
However, they are in the same equivalence class.
For all edges except $3a$ and $3b$, we have $f(e)=f'(e)$, so $\pi(f(e))=\pi(f'(e))$. For edges $3a$ and $3b$, 
\begin{align*}
f(3a)&=1b & f'(3a)&=1a & \pi(f(3a))&=\pi(f'(3a))=1 \\
f(3b)&=1a & f'(3b)&=1b & \pi(f(3b))&=\pi(f'(3b))=1
\end{align*}
so $f\equiv f'$. 
Applying $\pi$ and splitting the cycles gives
$$\Split(\pi(C_f))=\Split(\pi(C_{f'}))=
(4,5,6,7)(1,2,3)(1,2,3)(1,5,6,7,4,2,3).$$

More generally
in Fig.~\ref{fig:G->H}, let
$\nu_1=\nu_2=\nu_3=A$ and $\nu_4=\nu_5=\nu_6=\nu_7=B$.
Fig.~\ref{fig:G->H} shows $H$
 for
 $A=3$ and $B=2$.
The values of $\outdeg_G(x,\nu)$ clockwise from the
rightmost vertex are
$A,A+B,B,B,A+B$. We obtain
\begin{align*}
|\Multicyc_G(\vec\nu)| &=
\frac{A! (A+B)!^2 B!^2}{\prod_{i=1}^7 \nu_i!}
=\frac{A! (A+B)!^2 B!^2}{A!^3 B!^4}
=\frac{(A+B)!^2}{A!^2 B!^2}
=\binom{A+B}{A}^2 .
\end{align*}
\end{example}

Finally, we apply~(\ref{eq:|M_G(nu)|}) to count multicyclic de Bruijn sequences.
Let $G=G(1,q,k)$ and $H=G(m,q,k)$.
This has $\nu_e=m$ for all $e\in\Omega^k$.
Each vertex in $H$ has indegree and outdegree
 $mq$.
Thus,
\begin{equation}
|\Multicyc_G(\vec\nu)|
=
\frac{\prod_{x\in \Omega^{k-1}} (mq)!}{\prod_{\alpha\in\Omega^k} m!}
= \frac{(mq)!^{q^{k-1}}}{m!^{q^k}}
= W(m,q,k) \;.
\label{eq:MCDB_count_label_cycles}
\end{equation}
Each edge cycle $(e_1,\ldots,e_n)$ over $G$ yields a cyclic sequence
of length $n$ over $\Omega$ by taking the first (or last) letter of the $k$-mer labelling each edge. The edge cycle is aperiodic in $G$ iff the cyclic sequence of the cycle is aperiodic.
Thus, $|\MCDB(m,q,k)|$ is given by~(\ref{eq:MCDB_count_label_cycles}),
which agrees with Theorem~\ref{thm:MCDB_bijection}(b).

\appendix

\bibliographystyle{plain}
\bibliography{multidb-arxiv-refs}

\end{document}